\documentclass[sort&compress,3p]{elsarticle}
\journal{CAMWA} 
\usepackage{numcompress}
\usepackage[hidelinks]{hyperref}
\usepackage{lineno,hyperref}
\usepackage{lipsum}
\usepackage{amsfonts}
\usepackage{amsmath,amssymb}
\usepackage{amsthm}
\usepackage{fourier}
\usepackage{xcolor}
\usepackage{mathrsfs}
\usepackage{graphicx}
\usepackage{float}
\usepackage{subcaption}
\usepackage{tikz}
\usepackage[labelfont=bf]{caption}
\usepackage{epstopdf}
\usepackage[ntheorem]{empheq} 
\usepackage[final]{pdfpages}
\usepackage{subcaption}
\usepackage[justification=centering]{caption}
\usepackage{amsopn}
\usepackage[ruled,vlined]{algorithm2e}
\usepackage{array}
\usepackage{empheq}

\DeclareMathOperator*{\argmin}{arg\,min}

\newcommand{\R}{\mathbb{R}}
\newcommand{\PP}{\mathbb{P}}
\newcommand{\BB}{\mathbb{B}}
\newcommand{\Om}{\Omega}
\newcommand{\bOm}{\partial \Omega}
\newcommand{\Vn}{\textbf{n}}
\newcommand{\Vb}{\textbf{b}}

\newcommand{\Gb}{\beta}
\newcommand{\Gg}{\gamma}

\newcommand{\Gl}{\lambda}
\newcommand{\tnorm}[1]{{\left\vert\kern-0.25ex\left\vert\kern-0.25ex\left\vert #1\right\vert\kern-0.25ex\right\vert\kern-0.25ex\right\vert}}
\newcommand{\norm}[1]{{\left\vert\kern-0.25ex\left\vert #1\right\vert\kern-0.25ex\right\vert}}
\newtheorem{assmptn}{Assumption}
\newtheorem{lmm}{Lemma}
\newtheorem{prpstn}{Proposition}
\newtheorem{dfntn}{Definition}
\newtheorem{thrm}{Theorem}
\newtheorem{rmrk}{Remark}

\begin{document}

\begin{frontmatter}
\title{Adaptive stabilized finite elements via residual minimization onto bubble enrichments}

\author[one]{Jose G. Hasbani}\corref{mycorrespondingauthor}
\cortext[mycorrespondingauthor]{Corresponding author}
\ead{jose.hasbani@vistaenergy.com}

\author[two]{Paulina Sep\'ulveda}
\author[two]{Ignacio Muga}
\author[three]{Victor M. Calo}
\author[two]{Sergio Rojas}

\address[one]{Vista Energy, Argentina}
\address[two]{Instituto de Matem\'aticas, Pontificia Universidad Cat\'olica de Valpara\'iso, Valpara\'iso, Chile}
\address[three]{School of Electrical Engineering,  Computing and Mathematical Sciences,  Curtin University,  Bentley,  Australia}

\begin{abstract}
The Adaptive Stabilized Finite Element method (AS-FEM) developed in~\cite{ calo2020ASFEM} combines the idea of the residual minimization method with the inf-sup stability offered by the discontinuous Galerkin (dG) frameworks.  As a result,  the discretizations deliver stabilized approximations and residual representatives in the dG space that can drive automatic adaptivity. We generalize AS-FEM by considering continuous test spaces; thus, we propose a residual minimization method on a stable Continuous Interior Penalty (CIP) formulation that considers a $C^0$-conforming trial FEM space and a test space based on the enrichment of the trial space by bubble functions. {In our numerical experiments, the test space choice results in a significant reduction of the total degrees of freedom compared to the dG test spaces of~\cite{ calo2020ASFEM}} that converge at the same rate.  Moreover,  as trial and test spaces are $C^0$-conforming,  implementing a full dG data structure is unnecessary, simplifying the method's implementation considerably and making it appealing  for industrial applications, see~\cite{ Labanda2022}.

\end{abstract}


\end{frontmatter}

\tableofcontents
\section{Introduction}
The continuous Galerkin (cG) finite element methods (FEM) for advection-dominated reaction-type problems might suffer from instabilities. When the hyperbolic character of the problem becomes predominant, interior and outflow layers form, causing large local gradients in the solution. The cG-FEM is unstable in this regime, and various stabilized methods have been proposed. Techniques in the framework of cG-FEM are available in the literature, including Streamline-Upwind Petrov-Galerkin (SUPG) method~\cite{ johnson1984SUPG}, residual-free bubbles~\cite{ brezzi1999RFB}, and subviscosity models for advection-diffusion problems~\cite{ guermond1999artviz}. The relationships between the aforementioned strategies are also well understood in almost all cases. For example, in~\cite{ brezzi1992bstablerel}, a relation between stabilized finite element methods and the Galerkin method employing bubble functions was established for the advection-diffusion problems. In that work, the authors showed that bubble functions help stabilize the advective operator without using upwinding or any other numerical strategy. In particular, for the advection-diffusion problem, the Galerkin method employing piecewise linear functions enriched with bubble functions was shown to be equivalent to the SUPG method in the diffusive limit. Applications of this type of enrichment include stabilization of Stokes flow~\cite{ arnold1996mini} and stabilization of Galerkin approximation using artificial viscosity~\cite{ guermond1999artviz}. Hughes generalized the stabilized methods construction into a unified framework in the variational multiscale framework~\cite{ Hughes1995, Hughes1998}; these ideas were extended to other applications such as turbulence modeling~\cite{ Hughes:2017, Bazilevs2007}. Although those strategies are now reaching maturity, those methods have some drawbacks in certain complex flow regimes. For instance, the SUPG stabilization becomes non-symmetric and does not allow lumped mass, the residual free bubbles method adds additional degrees of freedom to the system, and the projection methods introduce hierarchical meshes for the projection on the subgrid viscosity model~\cite{ guermond1999artviz}.


Interior Penalty penalty methods using continuous functions were introduced originally by~\cite{ babushka1973InteriorPenalty, douglas1976InteriorPenalty} for different problems, namely, the biharmonic operator and the second-order elliptic and parabolic problems. These methods penalize the flux jump of the discrete solution at mesh interfaces. Thus, the method~\cite{ douglas1976InteriorPenalty} keeps the benefits of continuous finite element methods as they were standard for elliptic problems while simultaneously managing the difficulties encountered by these methods when the hyperbolic character of diffusion-advection problems becomes dominant in the advection limit. However, the robustness of the error estimate in the advection-dominated regime was not analyzed until~\cite{ burman2004edgebased}, in which the analysis was based on linear finite elements. Recently, a generalized hp-convergence analysis for a high-order Continuous Interior Penalty (CIP) finite element method was presented in~\cite{ burmanErn2007hpCIP} applied to advection-reaction and diffusion-advection-reaction problems.

Alternatively, the Continuous Interior Penalty (CIP) method, introduced in~\cite{ douglas1976intpenalty}, 
for advection-diffusion-reaction problems; this method adds an $L^2$ penalization to the flux jumps (i.e., gradient jumps for uniform coefficients) over the mesh interior edges/facets to stabilize a continuous approximation (see~\cite{ burman2004edgebased, burmanErn2005hpCIP, burman2006edgebased, burmanErn2007hpCIP}). In~\cite{ burmanErn2005hpCIP} and~\cite{ burmanErn2007hpCIP}, the authors developed a hp-convergence analysis for high-order CIP methods for advection-reaction and  advection-diffusion problems. This stabilization does not depend on the diffusion coefficient and considers the case of pure advection.  Moreover,~\citet{ burman2009errorCIP} introduced a formulation relating the stabilized continuous and discontinuous Galerkin frameworks with the CIP formulation. In that work, the author showed that both frameworks can be condensed into a single formulation, presented a robust a-posteriori error estimate for advection-reaction problems to guide adaptivity, and compared the low-order CIP and dG approximations using an adaptive approach showing that the CIP method achieves optimal convergence in the $L^2$ norm.


More recently, the Adaptive Stabilized Finite Element Method (AS-FEM) was introduced (see~\cite{ calo2020ASFEM}). This method formulates residual minimization problems within a dG mathematical framework. Namely, a discrete approximation of the solution in a continuous trial space is constructed by minimizing the residual in the dual norm of a dG test space with inf-sup stability. The residual minimization problem is equivalent to a saddle-point problem that inherits dG's inf-sup stability. There are some similarities with the Discontinuous Petrov-Galerkin (DPG) method as both minimize the residual in a non-standard norm (see, e.g.,~\cite{ ChaHeuBuiDemCAMWA2014, Calo2014, Dem2010, DemGopBOOK-CH2014, Dem2012, Dem2013, Niemi2011, Niemi2013, Niemi2013b, zitelli2011}). However, AS-FEM builds on non-conforming dG formulations, where stronger norms than those used in DPG may be chosen when the test space contains continuous functions. Application examples include diffusive-advective-reactive problems~\cite{ cier2020automatic}, incompressible Stokes flows~\cite{ kyburg2020, los2021dgirm}, continuation analysis of compaction banding in geomaterials~\cite{ cier2020adaptive} and weak constraint enforcement for advection-dominated diffusion problems~\cite{ Cier2020nonlinear}. 
The method has been successfully applied to several nonlinear problems, such as dynamic fracture propagation~\cite{ Labanda2022}, mineral deposition~\cite{ Poulet:2023}, and the method of lines for Bratu's equation~\cite{ Giraldo2023}.
In~\cite{ rojas2021GoA}, the authors extend AS-FEM to goal-oriented adaptivity (GoA); they describe a general theory for problems with well-posed formulations and provide error estimates to guide the GoA for advection-diffusion-reaction problems. In addition, they define a discrete adjoint system as a saddle-point problem where the discontinuous conforming space across element interfaces restricts the solution. The same dG inf-sup arguments guarantee the well-posedness of this adjoint saddle-point problem. Solving the primal and the adjoint problem requires the solution of a single saddle-point problem with two right-hand sides. Moreover,  the authors proposed two alternative stable discrete problems that can measure the discrete adjoint error of the problem; a strategy similar to a recent DPG theory~\cite{ keith2019goal} in which the adjoint problem is solved using the original saddle-point formulation with a different right-hand side.

The method proposed in~\cite{ calo2020ASFEM} and its extension to GoA in~\cite{ rojas2021GoA} contain a general theory motivated in a dG framework; this framework can also use continuous test spaces (see~\cite{ Labanda2022} for a demonstration of this idea). Herein, we analyze the extension of the AS-FEM based on a stable CIP formulation for an advection-reaction problem proposed in~\cite{ burman2004edgebased, burmanErn2007hpCIP}. We consider a bubble-enriched continuous trial space as a test space to explode the power of the residual minimization method, which delivers an error representative  which is robust and reliable to guide adaptive mesh refinements. This space is sufficient to guarantee a distance to the trial space and obtain a residual representative to drive adaptivity. 

We test the method's performance in a challenging advection-reaction problem. We choose this model problem since the elliptic character of the diffusion term in the advection-diffusion equation has some smoothing properties on the transport problem. Nonetheless, extending the results obtained in this work to advection-diffusion-reaction problems is straightforward. Additionally, we present a new result on a priori error estimates for the residual minimization method using continuous test spaces, which relies on an orthogonality argument for boundedness of the discrete bilinear form and coercivity that proves quasi-optimal convergence for advection-dominated problems. Finally, we perform adaptive numerical experiments using energy-based and GoA strategies~\cite{ calo2020ASFEM, rojas2021GoA}. In the energy-based examples,  we evaluate the performance of the residual error estimate obtained by the residual minimization method against the a-posteriori error estimate of~\cite{ burman2009errorCIP} regarding the relative error in the $L^2$-norm and the resulting adapted meshes.

The remainder of this paper is structured as follows: Section~\ref{sec:2} introduces the advection-reaction model problem. In Section~\ref{sec:3}, the notation and the discrete problem settings are presented.  In Section~\ref{sec:Residual}, we describe the CIP formulation in the context of bubble-enriched continuous spaces, its main properties, and the residual minimization method, and we state the main result of this work. In  Section~\ref{sec:num}, we present numerical experiments to show the method's performance and compare the results with other methods found in the literature. Finally, Section~\ref{sec:conclusions} summarizes the main contributions of this work and points to possible future research directions. 

\section{Model problem.}\label{sec:2}

Let $\Om \subset \R^{d}$ ( $d=2,3$) be an open bounded connected set, with Lipschitz boundary $\bOm$, and outward normal vector $\Vn$. Consider an advection field $\Vb\in [L^\infty(\Omega)]^d$  such that $\nabla \cdot \Vb\in L^\infty(\Omega)$, 
and a reaction coefficient $\mu \in L^\infty(\Omega)$. Let us define the following graph space:
\begin{align}
    W :=\left\{w \in L^2(\Omega):\, \Vb\cdot \nabla w \in L^2(\Omega)\right\},
\end{align}
equipped with the graph norm $\|w\|_W^2 = \|w\|_{L^2(\Omega)}^2 + \|\Vb\cdot \nabla w\|_{L^2(\Omega)}^2$. The advection field $\Vb$ partitions the boundary $\partial \Omega$ into 
inflow,  characteristic, and outflow parts,
having the following expressions when $\Vb$ is continuous\footnote{These boundaries may also be defined for the  general case $\Vb \in [L^\infty(\Omega)]^d$ and $\nabla \cdot \Vb \in L^\infty(\Omega)$; see  \cite[Section 2]{ broersen2018stability}.}: 
\begin{subequations}
\begin{align}
\partial\Omega_{-}& :=\{x \in \partial\Omega : \Vb(x)\cdot \Vn(x)<0\},\\
\partial\Omega_{0}& :=\{x \in \partial\Omega : \Vb(x)\cdot \Vn(x)=0\},\\
\partial\Omega_{+}& :=\{x \in \partial\Omega : \Vb(x)\cdot \Vn(x)>0\}.
\end{align}
\end{subequations}



Given $f \in L^2(\Omega) $, we consider the advection-reaction model problem that seeks $u\in W$ such that:
\begin{align}\label{eq:model_problem}
\left\{
    \begin{array}{l}
        \begin{array}{rl}
        \Vb \cdot \nabla u + \mu \, u = f, & \text{ in } \Om \smallskip\\
	u=0, & \text{ on } \bOm_-.
        \end{array}
    \end{array}\right.
\end{align}
%
%
We point out that traces of $W$ are well-defined\footnote{See, e.g.,~\cite[Section 2.1.3]{ Ern2012dG} and~\cite{ gopalakrishnan2015} for an  extension.}, implying that the linear space
 $W_{0,-}:=\{w\in W: w|_{\partial\Omega_-}=0\}$ is also well-defined. Moreover, equation~\eqref{eq:model_problem} translates into finding $u\in W_{0,-}$ such that $A u = f$, where $A: W_{0,-} \to L^2(\Om)$  (defined  as $A\,w := \Vb\cdot \nabla w +\mu w $, for all $w\in W_{0,-}$)  is a linear isomorphism  provided 
 \begin{equation}
    \mu - \dfrac{1}{2} \nabla \cdot \Vb \geq \mu_0 > 0, \quad \text{ \textit{a.e.} in } \Omega,
\label{cond:advection}
\end{equation}
holds for some positive constant $\mu_0$ (see~\cite[Proposition 5.9]{ ErnGermond2004} ).  When $\mu =0$ and $ \nabla\cdot \Vb=0$, then $A$ still is a linear isomorphism provided that $\Vb$ is an $\Omega$-filling field (see~\cite[Remark 5.10]{ ErnGermond2004}).
\begin{rmrk}[Non-homogeneous inflow boundary condition]
The trace operator is linear, continuous, and surjective from $W$ onto the space
$$
L^2(|\Vb\cdot \Vn|;\partial \Omega):=\left\{ w \text{ measurable in } \partial \Omega:  \int_{\partial \Omega}|\Vb\cdot \Vn| w^2<+\infty \right\},
$$
which allows us to consider non-homogeneous inflow boundary conditions in~\eqref{eq:model_problem}, whenever the inflow data belongs to
$L^2(|\Vb\cdot \Vn|;\partial \Omega)$ 
(see, e.g.,~\cite[Lemma 2.11]{ Ern2012dG}). 

\end{rmrk}
%
\section{Notation, discrete spaces and interpolation results.}\label{sec:3}

Let $\Omega_h$ be a quasi-uniform simplicial mesh of $\Om$, defined as a collection of finite many open connected elements $T\subset \Om$ with Lipschitz boundaries, such that  $\overline{\Omega}$ is the union of the closures of all mesh elements ${T}$ in $\Omega$. We denote by $\mathcal{F}^{0}_{h}$ the collection of all interior element boundaries (edges/faces); 
%
 by $\mathcal{F}^{\partial}_{h}$ the  collection of those  element boundaries that belong to $\bOm$; and set $\mathcal{F}_h:=\mathcal{F}^0_h \bigcup \mathcal{F}^{\partial}_h$. Given $T\in \Omega_h$, let us denote the diameter of $T$ by $h_T>0$, and let $h:=\underset{T\in\Om_h} {\max} \, h_T$.
 Analogously, for $e\subset \mathcal{F}^{0}_{h}$, let $h_e>0$ be the diameter $e$. 
%

For a given set $D\subset\mathbb R^d$, consider the Sobolev space $H^s(D)$ of order $s\geq 0$, and
 denote by $\|\cdot\|_{s,D}$ its well-known norm (see, e.g.,~\cite{ adams2003sobolev}). By convention, set $L^2(D)=H^0(D)$ and abbreviate the $L^2(D)$ inner product by $(\cdot,\cdot)_D$ and its respective norm by $\|\cdot\|_{D}$.
\begin{figure}[t]
\begin{center}
	\begin{tikzpicture}
	\draw (1,1) -- (4,0) -- (2,3) -- cycle;
	\draw (2,3)  -- (5,2) -- (4,0) --cycle;
	\draw[very thick] (2,3)  -- (4,0);
	\draw[->] (3,1.5) --(4,2);

	\node at (3.6,2) {$\Vn_e$};
	\node at (4,-0.2) {$e$};
	\node at (2.7,1) {$T^{+}$};
	\node at (4,1) {$T^{-}$};
	\end{tikzpicture}
\end{center}
\caption{Notation for an interface} \label{fig:interface}
\end{figure}
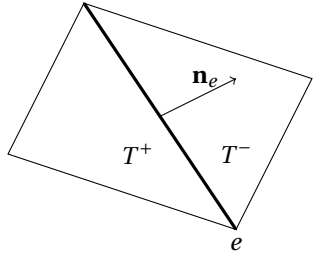

We define the \emph{broken} Sobolev space $H^s(\Om_h)$ as follows 
$$H^s(\Om_h):= \{ v \in L^2(\Omega): v|_{T} \in H^{s}(T) \text{ for all } T \in \Omega_h\},$$
with its corresponding broken norm $\|\cdot\|_{s,\Om_h}^2:= 
\displaystyle\sum_{T \in \Omega_h}\|(\,\cdot\,)|_T\|^2_{s,T}\,$.  %
%
When $s>\frac{1}{2}$, traces over the edges of elements are well-defined. 
Thus, for $e \subset \mathcal{F}^0_h$,  we define the jump of a function $v\in H^s(\Om_h)$ across $e$, as the following expression:
\begin{equation}\label{eq:scalar_jump}
	\llbracket v \rrbracket\big|_e := v\big|_{T^+} - v\big|_{T^-}, \qquad \forall e \in \mathcal{F}^{0}_h\,,
\end{equation}
where $v\big|_{T^+}$ and $v\big|_{T^-} $ are the traces over $e$ related with a predefined normal $\Vn_e$ (see Figure~\ref{fig:interface} for a reference of the interface notation). 

%

\subsection{Discrete spaces}

Let $\PP^p(T)$ be the space of polynomials of total degree $\leq p$ over  $T$. We consider the following discrete spaces:
\begin{equation}\label{eq:continuous_space}
	U^p_h :=\left\{ v \in C^0(\Om) : v|_T \in \mathbb{P}^p(T), \, \forall T\in \Om_h \right\},
\end{equation}
\begin{equation}\label{eq:broken_poly_space}
	V_h^p :=\left\{ v \in L^2(\Om) : v|_T \in \mathbb{P}^p(T), \, \forall T\in \Om_h \right\},
\end{equation}
and define the space of local bubble functions of degree $ \leq k$, with $k>d$ (see~\cite{ arnold1996mini}), by
\begin{equation}\label{eq:Bubble}
	\BB^k(\Om_h) := \left\{v \in H_0^1(\Omega): \, v|_T \in \PP^k(T) \bigcap H^1_0(T), \, \forall T \in \Om_h \right\}.
\end{equation}
Finally, we define the bubble-enriched continuous function space, with $k>\max\{p,d\}$, as 
 \begin{equation}\label{eq:BubbleSpace}
     U^{p, k}_h := U^p_h + \mathbb{B}^{k}(\Om_h).
 \end{equation}
 
As an example, in the case of triangular elements and  $k=3$, a local bubble function  over $T$ can be defined as a cubic function  spanned by $\lambda_1\lambda_2\lambda_3$, where $\Gl_1,\Gl_2,\Gl_3$ are the barycentric coordinates on $T$, see~\cite{ arnold1996mini, millar2021, Labanda2022}. 

\begin{rmrk}[Containment of discrete spaces]
When $k\leq p$ implies  $ U^{p, k}_h \equiv U^{p}_h $; alternatively, when $k>p$, the following discrete space contention occurs
  \begin{equation}
      U_h^p\subset U_h^{p,k}\subseteq U_h^{k}.
  \end{equation} 
\end{rmrk}

\subsection{Interpolation results}

Our convergence analysis follows~\cite{ burmanErn2007hpCIP} for the $hp$-continuous interior penalty (CIP) method, which uses trace and inverse-trace inequalities, local interpolation results, and error estimates for the $L^2(\Om)$-projection. Here is a summary of these results in two dimensions (i.e., $d=2$). 
We simplify notation by abbreviating the inequalities $ a \leq C\, b$ as $a \lesssim b$ whenever the positive constant $C$ is independent of the mesh and polynomial degree.

\begin{dfntn}[Admissible set] Given a simplex $T$, let $n_p := \dim\mathbb{P}^p(T)$. A nonempty set of nodes $\mathcal{A}=\{a_i\}_{1\leq i \leq n_p}$ of $T$ is admissible if and only if
$\mathcal A$ is unisolvent in $\mathbb{P}^{p}(T)$ and 
$\mathcal{A}\cap e$ is unisolvent in $\mathbb P^p(e)$, for all edges $e\subset\partial T$.
\end{dfntn}
Given a unisolvent set of nodes $\mathcal{A}$ of a simplex $T$, define:
$$
\mathbb{P}^p_{\mathcal{A}}(T) := \left\{ v \in \mathbb{P}^p(T) :  v(a_i)=0, \forall a_i \in \mathcal{A}\setminus\partial T\right\}.
$$

\begin{assmptn}\label{ass:reference}
 Let $\hat T$ be the reference element; for $\mathbb P ^p(\hat T)$, there exists an admissible set of nodes $\hat{\mathcal A}$ of $\hat T$, such that 
    \begin{align*}
        \| v\|_{\hat T} \lesssim p^{-1/2}\|v\|_{\partial \hat T}\,, \qquad \forall v \in \mathbb{P}_{\hat{\mathcal{A}}}^p(\hat T).
    \end{align*}
\end{assmptn}
In~\cite[Section~5.1]{ burmanErn2007hpCIP}, the authors establish that Assumption~\ref{ass:reference} leads to the following inequalities.
\begin{lmm}[Trace and inverse trace inequalities on triangles]
\label{lmm:hp-trace_ineq}
Under Assumption 1, 
 the following inequalities hold for any $T\in \Om_h$:
\begin{equation}\label{eq:trace-inequality}
\norm{v}_{\partial T} \lesssim \left(\dfrac{p^2}{h_T} \right)^{\frac{1}{2}} \norm{v}_T, \quad
  \forall v \in \mathbb{P}^p (T); 
  \end{equation}
\begin{equation}\label{trace2}
\norm{v}_{T} \lesssim \left(\dfrac{h_T}{p} \right)^{\frac{1}{2}} \norm{v}_{\partial T}, \quad\forall v \in \mathbb{P}^{p,0} (T),
\end{equation}
where $\mathbb P^{p,0}(T)$ denotes the subspace of $ \mathbb P ^{p}(T)$ spanned by those polynomials vanishing at all interior nodes of $T$.
\end{lmm}

%
 
\begin{dfntn}[Oswald Interpolation] Given $T\in\Omega_h$, let $\mathcal A_T$ be the image of the admissible set $\hat{\mathcal A}$ (see Assumption~\ref{ass:reference}) under an affine transformation mapping from $\hat T$ onto $T$. For each node $a\in\mathcal A_T$, consider the patch of elements
 $\Omega_{a}:= \big\{ T \in \Om_h, a \in \overline{T}  \big\}$.
The Oswald interpolation operator $I_{Os}:V_h^p\to U_h^p$  (see, e.g.,~\cite[Section~5.5.2]{ burmanErn2007hpCIP}),  
is defined locally by setting
\begin{equation}
	I_{Os} (v_h) (a) := \dfrac{1}{|\Omega_{a}|} \sum_{T\in  \Omega_a} v_h\big|_T(a)\,, \qquad \forall v_h\in V_h^p,
 \label{eq:Oswald}
\end{equation}
where $|\Omega_a|$ stands for the cardinality of $\Omega_a$.
\end{dfntn}

The following Lemma (proved in~\cite[Lemma 5.3]{ burmanErn2007hpCIP}) establishes the interpolation error associated with $I_{Os}$. 
\begin{lmm}[Interpolation error]\label{hp_interp_error} 
	For all $T \in \Om_h$, the following estimate holds:
	\begin{equation*}
		\norm{v_h - I_{Os} v_h}_T \lesssim 	%
		\left(\dfrac{h_T}{p} \right)^{\frac{1}{2}} \sum_{e\in \mathcal{F}_T} 
		\norm{\llbracket v_h \rrbracket}_e\,,\qquad \forall v_h \in V^p_h,
	\end{equation*}
 where $\mathcal{F}_T=\left\{e\in \mathcal{F}_h: e\cap\overline{T}\neq\emptyset\right\}$.
\end{lmm}

The following Lemma (proved in~~\cite[Lemma 5.4]{ burmanErn2007hpCIP}) establishes an estimate for the $L^2(\Omega)$-projection error. 
\begin{lmm}[Error estimate for the $L^2(\Om)$-projection]\label{hp_error_estimate_L2}%
Let $~\Pi_h: L^2(\Om) \mapsto U^p_h$ be the $L^2(\Om)$-orthogonal projector onto $U^p_h$. For all $u \in H^s(\Om), s\ge 1$, the following estimates~\eqref{eq:estimateL2} and~\eqref{eq:estimateL2grad} in Lemma~\ref{hp_error_estimate_L2} hold:
\begin{equation} \label{eq:estimateL2}
    \norm{u -\Pi_h u}_{\Om} \lesssim p^{\frac{1}{4}} \left(\dfrac{h}{p}\right)^r \norm{u}_{r, \Om},
\end{equation}
\begin{equation} \label{eq:estimateL2grad}
    \norm{\nabla\left(u -\Pi_h u\right)}_{\Om} \lesssim p^{\frac{5}{4}} \left(\dfrac{h}{p}\right)^{r-1} \norm{u}_{r, \Om},
\end{equation}
with $r = \min(p+1, s)$.
\end{lmm}

\section{Residual minimization method onto bubble enriched test spaces}\label{sec:Residual}

%
\subsection{Preliminaries}

Assume that $\nabla \cdot \Vb = 0$ and  $\mu(x)> \mu_0>0$ (a.e. in $\Omega$); thus, fulfilling condition~\eqref{cond:advection}. Related to Problem~\eqref{eq:model_problem}, we consider the bilinear form $a:H^1(\Omega)\times H^1(\Omega)\to \mathbb{R}$ defined by
\begin{equation}
    a(v, w) := \big(\mu v, w \big)_{\Om} -\big(v, \Vb \cdot \nabla w\big)_{\Om} + \big(\Vb \cdot \Vn\,v, w \big)_{\partial\Om^+}, \quad \forall v, w \in H^1(\Omega), \label{eq:biform}
\end{equation}
and the \textit{hp}--CIP bilinear form $j_{h,k}:H^s(\Omega_h)\times H^s(\Omega_h)\to \mathbb{R}$ (with $s>\frac{3}{2}$) defined by
\begin{equation}
    j_{h,k}(v, w) := \sum_{e\in \mathcal{F}^0_h} \Gg_{h,k} \big( \llbracket \nabla v\cdot \Vn_e\rrbracket,\, \llbracket \nabla w \cdot \Vn_e\rrbracket \big)_e, \quad \forall v, w \in H^s(\Omega_h), \label{eq:jdef}
\end{equation}
where $\Gg_{h,k}:=\dfrac{h^2_e}{k^{\alpha}}\,  \| \Vb \cdot \Vn_e \|_{L^{\infty}(e)}$  is the stabilization parameter (see~\cite{ burman2006edgebased, burmanErn2007hpCIP}).
We determine the exponent $\alpha$ using the \textit{hp}--convergence analysis and $k$ is related to the particular polynomial order used in the discrete counterparts. 

 We define the following norm for functions $w \in  H^s(\Omega_h)$, with  $s>\frac{3}{2}$ (see~\cite[Eq. 8]{ burmanErn2007hpCIP}) :
\begin{equation}\label{eq:coercive_norm}
    \tnorm{w}^2_{h,k}:=\big\| \mu_0^{1/2} w \big\|^2_\Om + \dfrac{1}{2}\big\| |\Vb \cdot \Vn|^{1/2} w\big\|^2_{\bOm} + j_{h,k}(w, w), \qquad \forall\, w\in H^s(\Omega_h).
\end{equation}
%




\begin{lmm}[Coercivity]\label{lem:coercive} 
For $s>\frac{3}{2}$, the bilinear form  $b_h(\cdot,\cdot):=a(\cdot,\cdot) + j_{h,k}(\cdot,\cdot)$ is coercive in $H^s(\Omega_h)$ with respect to the norm defined in ~\eqref{eq:coercive_norm}.
\end{lmm}
\begin{proof}
Since $\nabla\cdot\Vb=0$, observe that $\big(w, \Vb \cdot \nabla w\big)_{\Om}=\frac{1}{2}\big(\Vb\cdot \Vn w,w\big)_{\partial\Omega}$, for all $w\in H^s(\Omega_h)$. Hence,
$$
a(w,w)=\big(\mu w,w\big)_\Omega-\frac{1}{2}\big(\Vb\cdot \Vn w,w\big)_{\partial\Omega}
+ \big(\Vb\cdot \Vn w,w\big)_{\partial\Omega^+}=\big(\mu w,w\big)_\Omega
+\dfrac{1}{2}\big\| |\Vb \cdot \Vn|^{1/2} w\big\|^2_{\bOm}\,.$$
Thus, since $\mu(x)>\mu_0$, we get
$$
b_h(w,w)=a(w,w)+j_{h,k}(w,w)\geq \tnorm{w}^2_{h,k},\quad\forall w\in H^s(\Omega_h). 
$$
\end{proof}

\subsection{Continuous interior penalty method onto bubble enriched continuous space.}\label{sec:5}

Consider the CIP formulation  in the  bubble enriched continuous spaces $U_h^{p,k}$ as follows,
\begin{equation}\label{eq:CIP_ discrete_bubbles}
\left\{\begin{array}{l}
	\text{Find } \theta_h \in U^{p,k}_h \text{ such that:}\\
	 b_h(\theta_h, \nu_h) =  l_h(\nu_h), \quad \forall \, \nu_h \in U^{p,k}_h,
\end{array}
\right.
\end{equation}
where $b_h(\theta_h, \nu_h)$  is the sum of the $a$ and $j$ forms defined in~\eqref{eq:biform} and~\eqref{eq:jdef}, and $l_h(\nu_h)$ is $(f,\nu_h)_\Omega$.

In the following, we assume that  $u \in H^s(\Om), ~s> \frac{3}{2}$ solves~\eqref{eq:model_problem} and $\theta_h\in U^{p,k}_h$ solves~\eqref{eq:CIP_ discrete_bubbles}. Thus, the formulation~\eqref{eq:CIP_ discrete_bubbles} satisfies the following properties,
\begin{lmm}[Coercivity]\label{lmm:Stability_bubbles}
    For all $\nu_h \in U^{p,k}_h$ then,
    \begin{equation}
        b_h(\nu_h, \nu_h) \gtrsim \tnorm{\nu_h}^2_{k,h}.
    \end{equation}
    \begin{proof}
        Follows from Lemma~\ref{lem:coercive}.
    \end{proof}
\end{lmm} 
\begin{lmm}[Consistency]\label{lmm:concistencyBubbles} 
Let $u \in H^s(\Om), s>\frac{3}{2}$ solve~\eqref{eq:model_problem} and let $\theta_h\in U_h^{p,k}$ solve~\eqref{eq:CIP_ discrete_bubbles}. Then, for all $v_h \in U^{p,k}_h$,
\begin{equation}
    b_h(u-\theta_h, v_h) = 0.
\end{equation}
\end{lmm}
\begin{proof}
    Since $u\in H^s(\Omega)$, $q>\frac{3}{2}$, the jump $[\nabla u\cdot n]_e=0$, for all $e\in \mathcal{F}$. Thus, $j_{h,k}(u,v_h)=0$ for all $v_h\in U^{p,k}$.
    Thus,
    \begin{align*}
       b_h(u-\theta_h, v_h) & = b_h(u, v_h)-b(\theta_h,v_h)\\
       & = a(u, v_h)-(f,v_h)=0.
    \end{align*} 
\end{proof}

\subsection{Residual minimization}
\citet{ calo2020ASFEM} present the Adaptive Stabilized Finite Element Method (AS-FEM), which combines residual minimization with an inf-sup stable discretization (e.g., a discontinuous Galerkin (dG) formulation) to deliver a robust on-the-fly adaptive method. Consequently, this method delivers a stabilized approximation of the solution and a residual representative in the dG space that drives adaptivity with no further a-posteriori error analysis.
In the following, we enrich the discretization space with bubbles $U_h^{p,k}$, see~\eqref{eq:BubbleSpace}, and use it as a test space; a choice that guarantees a distance between the trial and the test space, which is suitable for residual minimization and adaptive mesh refinement guided by the built-in error representative. We adopt the CIP formulation~\eqref{eq:CIP_ discrete_bubbles}, which is coercive-stable, and show a new result inspired by~\cite{ burmanErn2007hpCIP} for the a-priori error estimate for the residual minimization method that relies on boundedness and discrete coercivity properties of $b_h(\cdot, \cdot)$. In the abstract setting of AS-FEM, we consider two real Hilbert spaces $U, V$,  and a conforming subspace $U^p_h$ of either $U$ or $V$. In addition, we assume that the discrete variational formulation satisfies Lemmas~\ref{lmm:Stability_bubbles} and~\ref{lmm:concistencyBubbles}. Thus, the main idea behind AS-FEM is finding $u_h \in U^p_h$ using a well-posed discrete variational formulation (e.g., discrete problem~\eqref{eq:CIP_ discrete_bubbles}) set in a discrete space, which can be chosen as $U^{p,k}_h$, and find a residual representative of the error in the $\tnorm{\cdot}_{h,p}$ via residual minimization in the discrete-dual space of that discrete space $\left(U^{p,k}_h\right)^*$.

Let
\begin{equation}\label{eq:B_op}
\left\{\begin{array}{l}
	B_h : U^{p,k}_h \mapsto \left(U^{p,k}_h\right)^*\\
	\\
	\qquad w_h \mapsto b_h(w_h, \cdot)
\end{array}
\right. ,
\end{equation}
and let $R^{-1}_h$ be the inverse of the Riesz map defined as:
\begin{equation}\label{eq:Riesz}
\left\{\begin{array}{l}
	R_h : U^{p,k}_h \mapsto \left(U^{p,k}_h\right)^{*}\\
	\\
	\qquad \, \big\langle R_h \tau_h, \nu_h \big\rangle_{\left(U^{p,k}_h\right)^* \times U^{p,k}_h} := (\tau_h, \nu_h)_{U^{p,k}_h}
\end{array}
\right..
\end{equation}
In addition, the dual triple norm $\tnorm{\cdot}_{\left(U^{p,k}_h\right)^*}$ is
\begin{equation}\label{eq:dual_norm}
\tnorm{\tau}^*_{k,h} := \underset{0\neq \nu_h \in U^{p,k}_h}{\sup} \dfrac{\big\langle \tau, \nu_h \big\rangle_{\left(U^{p,k}_h\right)^*  \times U^{p,k}_h}}{\tnorm{\nu_h}_{k,h}}, \quad \forall \tau \in \left(U^{p,k}_h\right)^{*}.
\end{equation}

Then, instead of solving the discrete problem~\eqref{eq:CIP_ discrete_bubbles}, we solve the following minimization problem,
\begin{subequations}
\begin{empheq}[left=\empheqlbrace]{align}
\text{Find } & u_h \in U_h^p\subset U,  \text{ such that :} \nonumber\\
u_h &=\argmin_{w_h\in U_h^{p}} \frac{1}{2}\tnorm{l_h(\cdot)-b_h(w_h,\cdot)}^2_{\left(U_h^{p,k}\right)^*} = \argmin_{w_h\in U_h^p}\frac{1}{2}\tnorm{R_h^{-1}B(\theta_h-w_h) }^2_{U_h^{p,k}}
\end{empheq}
\label{eq:resmin}
\end{subequations}
where $\theta_h \in U^{p,k}_h$ is the solution of the discrete variational problem~\eqref{eq:CIP_ discrete_bubbles}.

Problem~\eqref{eq:resmin} is equivalent to the following saddle-point problem (see~\cite{ CohDahWelM2AN2012}).
\begin{subequations}
\begin{empheq}[left=\empheqlbrace]{align}
   \text{Find }  (\varepsilon_h, u_h) \in U^{p,k}_h \times U^p_h, & \text{ such that:} \nonumber\\
        (\varepsilon_h  ,  \nu_h)_{U^{p,k}_h} + b_h(u_h \, ,  \nu_h ) & =   l_h(\nu_h),   &\quad \forall\, \nu_h \in U^{p,k}_h, \label{eq:residual-a}\\
        b_h(w_h \, , \, \varepsilon_h) & = \  0,  &\quad \forall\, w_h \in U^p_h, \label{eq:residual-b}
  \end{empheq}
  \label{eq:residual}
\end{subequations}
where $\varepsilon_h \in U^{p,k}_h$ is the residual representative in terms of $\theta_h \in U^{p,k}_h$ and $u_h \in U^p_h$ is the minimizer of the quadratic functional~\eqref{eq:resmin}. Solving the saddle-point problem~\eqref{eq:residual} has several interesting properties. Firstly, the system inherits the discrete stability of the weak formulation~\eqref{eq:CIP_ discrete_bubbles}. Secondly, the residual representative $\varepsilon_h$ is an efficient and robust error estimate for automatic adaptivity.

We now state an a-priori error estimate for the residual minimization problem~\eqref{eq:residual}, our main result. 

\begin{assmptn}\label{Assump1} Let $u \in H^s(\Om), s>\frac{3}{2}$ solve~\eqref{eq:model_problem} and let $u_h\in U_h^{p}$ solve~\eqref{eq:residual}.

Then, there exists a continuous operator $\Pi_h:U\to U^p_h$, such that 
\begin{subequations}
\label{as:projector}
    \begin{align}
        b(u-\Pi_h u,v_h) &\lesssim\tnorm{u-\Pi_h u}_{h,k,\#}\tnorm{v_h}_{h,k}, \quad \forall \, v_h \in U_h^{p,k} \label{eq:projector_property}
    \end{align}
where convergence rates result from the bound on $\tnorm{u-\Pi_h u}_{h,k,\#}$ in terms of $h,p$ with respect to the $\|u\|_q$ norm, with $q=\min(p+1,s)$.
\end{subequations}
    
\end{assmptn}

\begin{thrm}[AS-FEM a priori error estimate onto bubble enrichment]\label{thrm:a-priori-err-estim-min-res}
Let $u \in H^s(\Om), s>\frac{3}{2}$ solve~\eqref{eq:model_problem} and let $u_h\in U_h^{p}$ solve~\eqref{eq:residual}. Suppose that the discrete formulation~\eqref{eq:CIP_ discrete_bubbles} satisfies Lemma~\ref{lmm:Stability_bubbles} and Assumption~\ref{Assump1}, then the following estimate holds,
\begin{equation}\label{eq:a-priori-min-res}
	\tnorm{u - u_h}_{h,k} \lesssim \tnorm{u - \Pi_h u}_{h,k,\#},
\end{equation}
 where $~\Pi_h: U \mapsto U^{p}_h$ is defined as in Assumption~\ref{Assump1}. Therefore, 
 $\tnorm{u - u_h}_{h,k}$ can be bounded in terms of $h,p,$ and 
$\|u\|_q$ norm, with $q=\min\{p+1,s\}$.
\end{thrm}
\begin{proof}
Using triangular inequality and the inf-sup condition that holds from the coercivity condition. 
\begin{align} 
    \tnorm{u - u_h}_{h,k} 
    & \leq \tnorm{u-\Pi_h u_h}_{h,k} + \tnorm{\Pi_h u-u_h}_{h,k}.\\
    & \leq \tnorm{u-\Pi_h u_h}_{h,k} + \sup_{v_h\in U_h^p\setminus\{0\}}\frac{b\left(u_h-\Pi_h u, v_h\right)}{\tnorm{v_h}_{h,k}}.\label{eq:ineq}
\end{align}
We now bound the last term of~\eqref{eq:ineq}. Let $\theta_h$ solve equation~\eqref{eq:CIP_ discrete_bubbles}, the consistency condition~\eqref{lmm:Stability_bubbles} establishes:
\begin{align}
l_h(v_h) =b_h(\theta_h, v_h)=b_h(u, v_h) \qquad \forall v_h\in U^{p,k}. \label{eq:condition} 
\end{align}
Since $u_h$ solves~\eqref{eq:residual},  the first equation in~\eqref{eq:residual} and~\eqref{eq:condition} imply 
\begin{align*}
    b(u_h-\Pi_hu,v_h) = b(u-\Pi_h u, v_h) - (\epsilon_h,v_h)_{h,k}.
\end{align*}
Then, 
\begin{align*}\sup_{v_h\in U_h^p\setminus\{0\}}\frac{b\left(u_h-\Pi_h u, v_h\right)}{\tnorm{v_h}_{h,k}} & \lesssim \sup_{v_h\in U_h^p\setminus\{0\}}\frac{b\left(u-\Pi_h u, v_h\right)}{\tnorm{v_h}_{h,k}} +  \sup_{v_h\in U_h^p\setminus\{0\}}\frac{b\left(\epsilon_h,v_h\right)}{\tnorm{v_h}_{h,k}}\\ & \lesssim
{\tnorm{u-\Pi_h u}_{h,k,\#}}+ \sup_{v_h\in U_h^p\setminus\{0\}}\frac{\left(\epsilon_h,v_h\right)_{h,k}}{\tnorm{v_h}_{h,k}} & 
\text{(using property~\eqref{eq:projector_property})} \\
& \lesssim {\tnorm{u-\Pi_h u}_{h,k,\#}} + {\tnorm{\epsilon_h}_{h,k}}
\end{align*}
Moreover, 
\begin{align*}
    \tnorm{\epsilon_h}_{h,k}^2& =l_h(\epsilon_h)-b_h(u_h,\epsilon_h) & \text{(using~\eqref{eq:residual-a} )}\\
    &= b_h(u,\epsilon_h) &  \text{(by~\eqref{eq:condition} and~\eqref{eq:residual-b} )} \\
    & = b_h(u-\Pi_h u ,\epsilon_h) & \text{( by~\eqref{eq:residual-b})}
\end{align*}
Thus, 
\begin{align*}
    \tnorm{\epsilon_h}_{h,k} = \frac{(\epsilon_h,\epsilon_h)_{h,k}}{\tnorm{\epsilon_h}_{h,k}}= \frac{b(u-\Pi_h u, \epsilon_h)}{\tnorm{\epsilon_h}_{h,k}}\lesssim \tnorm{u-\Pi_h u}_{h,k,\#}.
\end{align*}
\end{proof}

\noindent As~\cite{ calo2020ASFEM, rojas2021GoA} state, the residual representative $\varepsilon_h \in U^{p,k}_h$ is an efficient error estimate in the energy norm.
%





\begin{prpstn}\label{prop:bubble}
Consider $k=p+1>d=2$, with $d$ the dimension of the space, $u\in H^s(\Omega)$ solution of~\eqref{eq:model_problem}. 

Then, Assumption~\ref{Assump1} holds by taking the continuous $L^2(\Omega)$ projection $\Pi_h:U\to U^p_h$,
\begin{align*}
    b(u-u_h,v_h) &\lesssim \tnorm{u-\Pi_h u}_{h,k,\#}\|v_h\|_h, \quad \forall \, v_h \in U_h^{p,k}, \label{eq:projector_property}
\end{align*}
with norm:
$$\tnorm{u}_{ h,k, \#} := \tnorm{u}_{h, k} + \big(k^2 \, h^{\frac{1}{2}}p^{\frac{1}{2}} + k^{\frac{\alpha}{2}}p^{-\frac{1}{2}} \big)|u|_h .$$
and the semi-norm $|\cdot|_h,$
\begin{equation}
|\eta|_h :=\left(\sum_{T\in \Omega_h} h^{-1}_T\beta_{T,\infty} \|\eta\|^2_T\right)^{1/2} 
\label{eq:semi-norm}
\end{equation}
Taking $\alpha=7/2$,  we also obtain:
\begin{align*}
    \|u-\Pi_h u \|_{h,k}\lesssim
    ( p+p^{\frac{5}{2}}h^{\frac{1}{2}})\left(\frac{h}{p}\right)^{q-1/2}    \|u\|_{q,\Omega}
\end{align*}
with $q=\min\{p+1,s\}$. 
\end{prpstn}
\begin{proof}
See~\ref{Appendix}.
\end{proof}

\begin{assmptn}[Saturation]\label{assmptn: saturation_min_res}
Let $u\in U$ be the solution of~\eqref{eq:model_problem} and $\theta_h \in U^{p,k}_h$ be the discrete solution of~\eqref{eq:CIP_ discrete_bubbles} and let $u_h \in U^p_h$ be the solution of the saddle-point problem~\eqref{eq:residual}. There exists a constant $C_{\text{s}} \in [0, 1)$, uniform with respect of the mesh size such that,
\begin{equation*}
    \tnorm{u - \theta_h}_{h,k} \leq C_{\text{s}} \tnorm{u - u_h}_{h,k}.
\end{equation*}
\end{assmptn}
This assumption states that the discrete solution $\theta_h$ is closer than $u_h$ to the exact solution $u$ with respect to the norm $\tnorm{\cdot}_{h,k}$. This is a relevant assumption as  $U^p_h \subset U^{p,k}_h$. However, this assumption does not necessarily hold in the pre-asymptotic regime (see~\cite{ calo2020ASFEM}) or if $~U^{p,k}_h$ is not rich enough. Additionally, we assume that the error estimate in~\eqref{eq:a-priori-min-res} is at least quasi-optimal in the following sense:
\begin{assmptn}[Optimality and quasi-optimality]\label{assmptn:optimality}
If the analytical solution $u$ is sufficiently regular, the quantities $\tnorm{u-u_h}_{h,k}$ and $\tnorm{u - \Pi^{p}_h u_h}_{h,k,\#}$ follow the same convergence rate as $h \rightarrow 0^+$. If the norms $\tnorm{\cdot}_{h,k}$ and $\tnorm{\cdot}_{h,k,\#}$ are equal, then the error estimate~\eqref{eq:a-priori-min-res} is optimal, otherwise, it is quasi-optimal.
\end{assmptn}
\noindent The robustness of the residual representative and the a posteriori error estimate is given by the following result:
\begin{prpstn}[Robustness of the residual representative and a posteriori error estimates]\label{prpstn:robustness_residual_repr}
Considering the same assumptions of Theorem~\ref{thrm:a-priori-err-estim-min-res} and the triple norm in $U^{p,k}_h$ defined in~\eqref{eq:coercive_norm}, the following bounds for $\tnorm{\varepsilon_h}_{h,k}$ hold:
\begin{equation}\label{eq:bounds_error_estimate}
    \tnorm{\theta_h - u_h}_{h,k} \lesssim \tnorm{\varepsilon_h}_{h,k} \lesssim \tnorm{u- \Pi_h^p u}_{h,k,\#},
\end{equation}
where, $\theta_h$ is the solution of the discrete problem~\eqref{eq:CIP_ discrete_bubbles}, $\Pi_h^p: L^2(\Om) \mapsto U^{p,k}_h$ is the $L^2(\Om)$-orthogonal projection onto $U^{p,k}_h$ and $u$ is the exact solution of the continuous problem~\eqref{eq:model_problem}. Moreover, if the solution satisfies Assumption~\ref{assmptn: saturation_min_res}, the following efficiency error estimate holds:
\begin{equation}\label{eq:a_posteriori_error_estimate}
    \tnorm{\varepsilon_h}_{h,k}  \lesssim \tnorm{u-u_h}_{h,k} .
\end{equation}
Additionally, if Assumption~\ref{assmptn:optimality} is satisfied, $\tnorm{\theta_h - u_h}_{h,k}$ and $\tnorm{\varepsilon_h}_{h,k}$ have the same convergence rate as $h \rightarrow 0^+$.
\end{prpstn}

\section{Numerical experiments.}\label{sec:num}
In this section, we consider several numerical experiments in the context of the advection-reaction equation to demonstrate the performance of the residual minimization method considering bubble-enriched test spaces. The method can be extended to advection-diffusion-reaction problems, modifying the formulation and the discrete norm~\eqref{eq:coercive_norm} to the one defined in~\cite{ burmanErn2007hpCIP}.
In our first numerical example, we compare the performance of the residual-based error estimate to guide adaptivity against the a posteriori estimate proposed in~\cite{ burman2009errorCIP} considering the same initial problem setting and element marking strategy. As a second numerical example, we compare the results for the goal-oriented adaptivity (GoA) strategy proposed in~\cite{ rojas2021GoA} using a dG framework against those obtained using the CIP formulation and residual minimization using a bubble-enriched test space. Although in~\cite{ rojas2021GoA}, the results were obtained using a different formulation, the comparison using continuous test spaces is still fair since we calculate the relative error of a quantity of interest (QoI).

In all our numerical experiments, we use a standard adaptive procedure that considers an iterative process in which each iteration consists of the following four steps:
\begin{equation}\label{eq:adaptive_strategy}
    \textrm{SOLVE} \rightarrow \textrm{ESTIMATE} \rightarrow \textrm{MARK} \rightarrow \textrm{REFINE}.
\end{equation}
In addition, we adopt the Döfler bulk-chasing criterion (see~\cite{ dorfler1996}), where elements are marked when the local error estimate value is above a fraction of the total estimated error. Following~\cite{ calo2020ASFEM, rojas2021GoA}, we adopt 50 \%  error fraction for the energy-based adaptivity and a 20 \% error fraction for the GoA strategy. We also employ a bisection-type refinement criterion~\cite{ bank1983refinementAlgo} for the adaptive solver. Finally, we use Fenics~\cite{ fenics2015} Python library to implement each of the numerical examples, and we use a direct LU solver for the resulting linear systems.
\subsection{Advection-reaction problem.}
We consider the advection-reaction problem~\eqref{eq:model_problem} over the unit square domain $\Om = [0,1]^2 \subset \mathbb{R}^2$. Following~\cite{ burman2009errorCIP}, we take the reaction parameter $\mu=0.1$, the velocity field
\begin{equation*}
	\Vb(x_1,\, x_2) = \left( \dfrac{x_2+1}{\sqrt{x^2_1 + (x_2+1)^2}}, \dfrac{-x_1}{\sqrt{x^2_1 + (x_2+1)^2}} \right)^T
\end{equation*}
and $g(x_1, x_2)$ so that the exact solution reads as,
\begin{equation*}
    u(x_1, x_2) = e^{\mu \sqrt{x^2_1+(x_2+1)^2} \arcsin\left(\frac{x_2+1}{\sqrt{x^2_1+(x_2+1)^2} }
    \right)} \arctan\left(\dfrac{\sqrt{x^2_1+(x_2+1)^2} - 1.5}{\delta}\right)
\end{equation*}
where $\delta$ is a parameter that controls the stiffness of the internal boundary layer (see Figure~\ref{fig1:advection-advection problem}).
In this example, we set $\delta = 0.01$ to obtain a smooth solution to assess the expected convergence rates. In addition, we apply the adaptive strategy from~\eqref{eq:adaptive_strategy} to solve this problem using the a-posteriori residual estimate of~\cite{ burman2009errorCIP} and the residual minimization error estimate to guide the adaptivity from the same initial mesh (see Figure~\ref{fig_3a:initial_mesh}).
\begin{figure}[H] 
    \centering
    \begin{subfigure}[b]{0.3\textwidth}
        \centering
        \includegraphics[width=\textwidth]{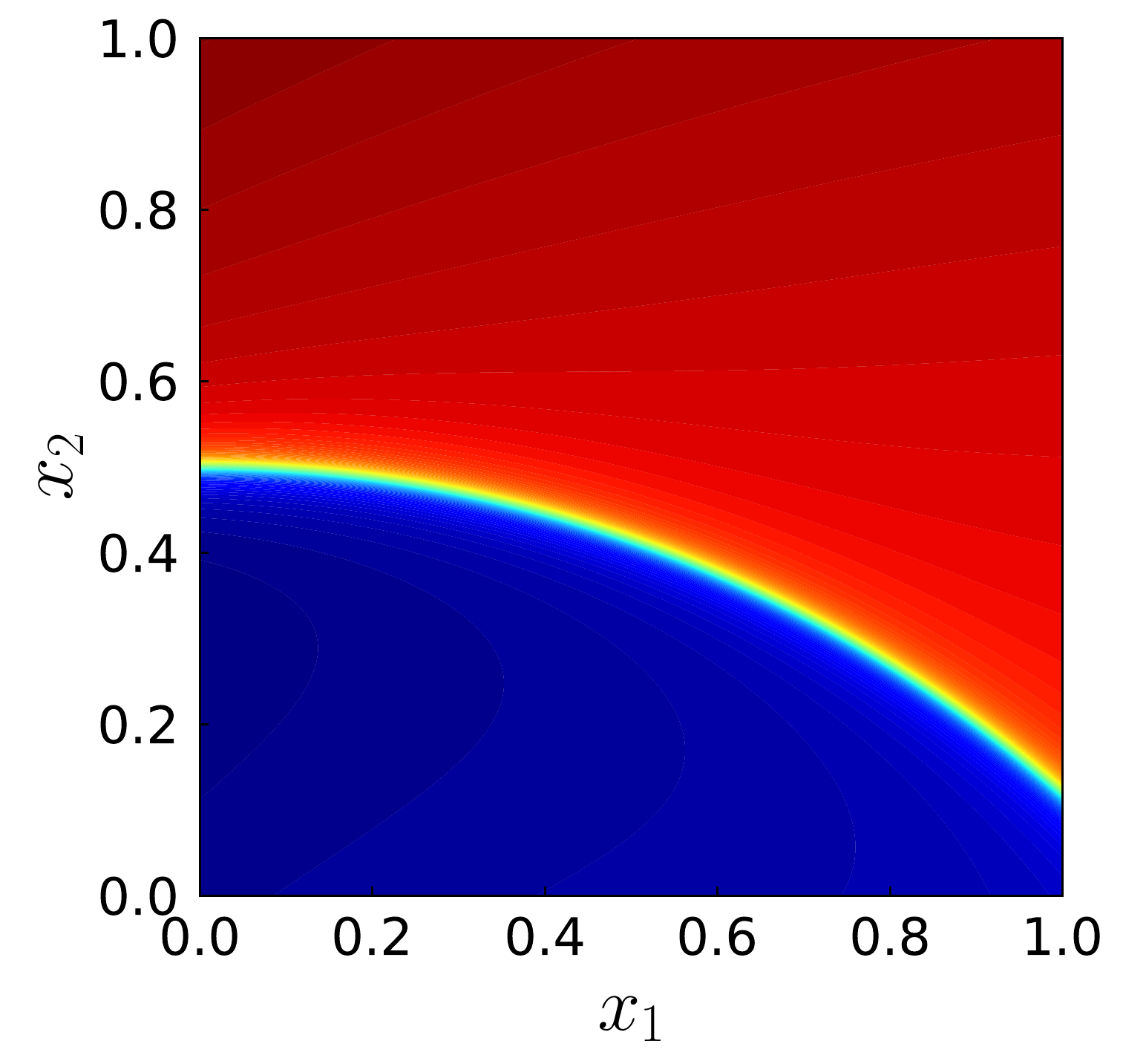}
        \caption{$\delta=0.01$}
    \end{subfigure}
    \hspace{0.2cm}
    \begin{subfigure}[b]{0.3\textwidth} 	 
        \centering
        \includegraphics[width=\textwidth]{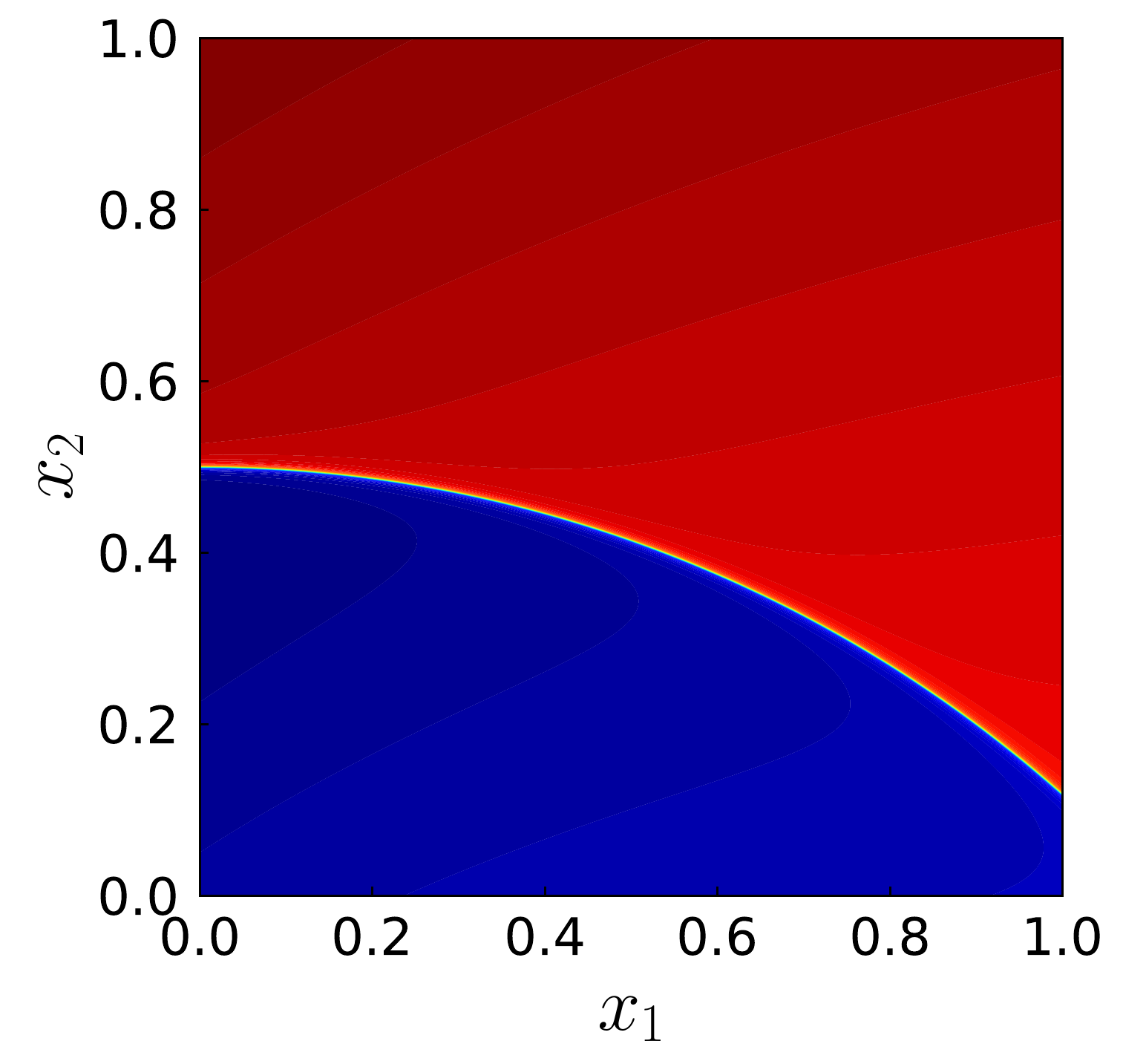}
        \caption{$\delta=0.001$}
    \end{subfigure}
    \hspace{0.2cm}
    \begin{subfigure}[b]{0.3\textwidth} 	 
        \centering
        \includegraphics[width=\textwidth]{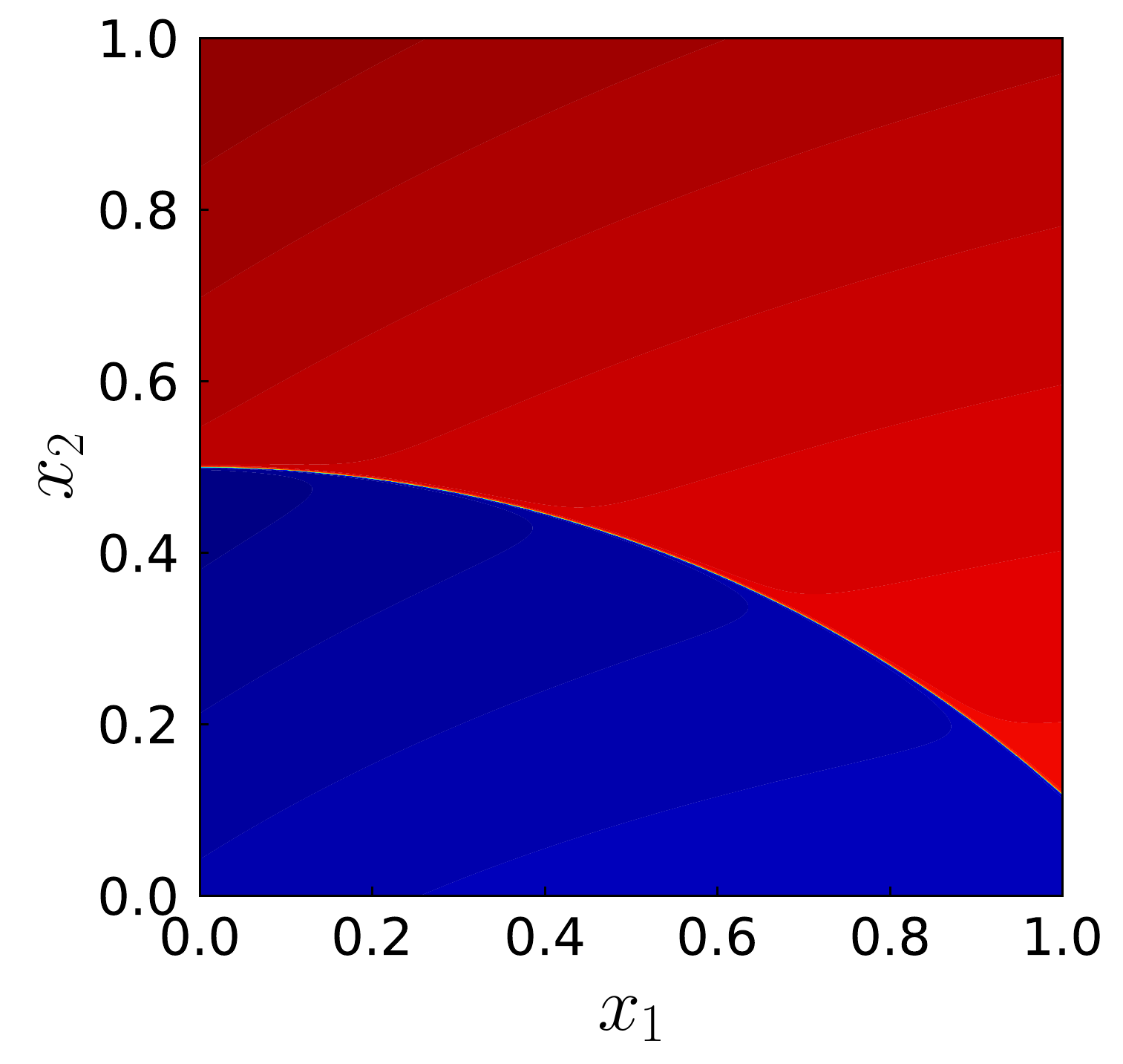}
        \caption{$\delta=0.0001$}
    \end{subfigure}
    \caption{2D advection-reaction problem (exact solution)}
    \label{fig1:advection-advection problem}
\end{figure}
Figures~\ref{fig_3b:refined_mesh_min_res} and~\ref{fig_3c:refined_mesh_burman} compare the refined meshes resulting from residual-based and a-posteriori estimates at similar numbers of degrees of freedom $(\textrm{DoFs})$. The a-posteriori estimate focuses the mesh refinement at the outflow boundary and the interior layer, with less emphasis at the inflow boundary. The residual-based estimator concentrates the refinement at the inflow boundary and progressively solves the interior layer.
\begin{figure}[H] 
    \centering
    \begin{subfigure}[b]{0.3\textwidth}
        \centering
        \includegraphics[width=\textwidth]{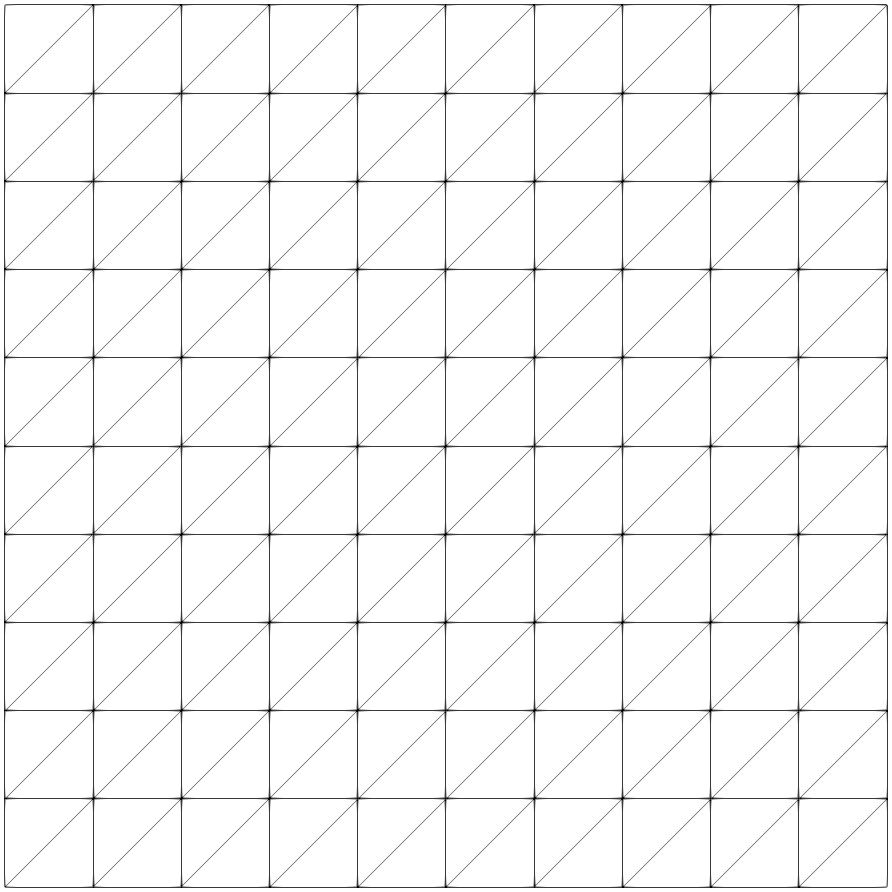}
        \caption{Initial mesh}
    \label{fig_3a:initial_mesh}
    \end{subfigure}
    \hspace{0.2cm}
    \begin{subfigure}[b]{0.3\textwidth} 	 
        \centering
        \includegraphics[width=\textwidth]{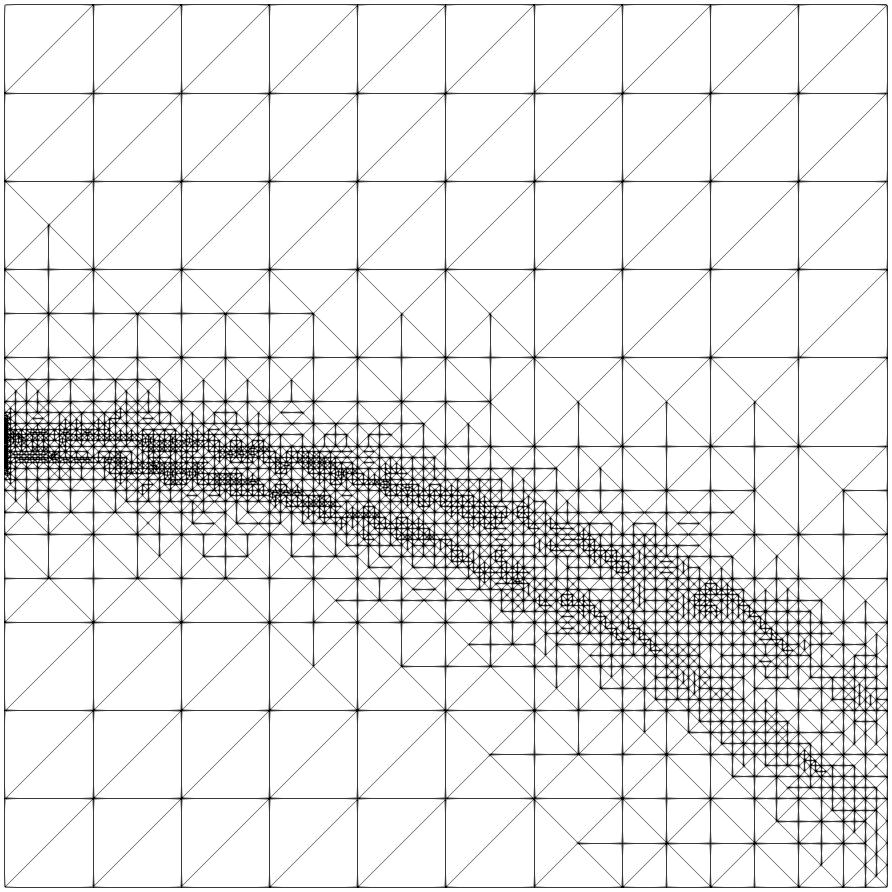}
        \caption{Residual-based estimator (143,270 $\textrm{DoF}$)}
        \label{fig_3b:refined_mesh_min_res}
    \end{subfigure}
    \hspace{0.2cm}
    \begin{subfigure}[b]{0.3\textwidth} 	 
        \centering
        \includegraphics[width=\textwidth]{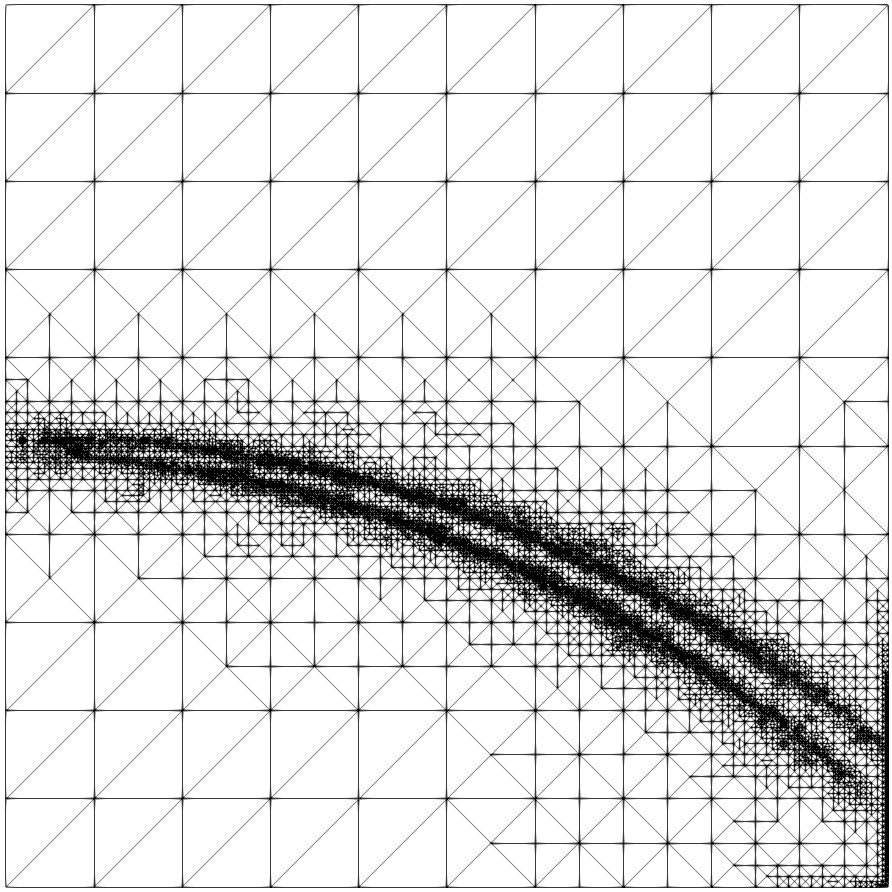}
        \caption{A-posteriori estimator (145,720 $\textrm{DoF}$)}
        \label{fig_3c:refined_mesh_burman}
    \end{subfigure}
    \caption{Initial \& refined meshes for similar total degrees-of-freedom (DoF) number using residual-based \& a-posterior estimator from~\cite{ burman2009errorCIP}}
    \label{fig3:refined_meshes_num_ex_1}
\end{figure}
Figure~\ref{fig2:comparison_min_res_burman} shows the comparison of the $L^2$-relative error between the a posteriori error estimate proposed in~\cite{ burman2009errorCIP} (red curve) and the residual-based estimate (blue curve). For the residual minimization scheme, we choose a piecewise continuous trial space of order $p=1$  and its enrichment with bubbles of order $k=p+2$ as a test space. As for the a-posteriori error estimation strategy, we use conforming piecewise continuous trial and test space of order $p=1$~\cite{ burman2009errorCIP}. We plot all quantities against the total number of degrees of freedom ($\textrm{DoFs}$), and the triangle shows the slope $\textrm{DoFs}^{-1}$, which is the optimal convergence in the $L^2$-norm for smooth problems~\cite{ burman2009errorCIP}. The convergence rate of the CIP formulation that uses the residual minimization strategy over a bubble-enriched test space is slightly better than the one obtained from applying the CIP method with the a-posteriori error estimation strategy. In addition, note, in particular, the superconvergence of the $L^2$-error using CIP and residual minimization method when the mesh enters the asymptotic range. This is meaningful since the residual error estimate is defined in the energy norm. Although the convergence rate is improved in the $L^2$-norm the total number of $\textrm{DoFs}$ using residual minimization is greater than the $\textrm{DoFs}$ using the strategy in~\cite{ burman2009errorCIP}. This fact is because residual minimization solves a full saddle-point system that gives a discrete solution and a residual estimate in the bubble-enriched test space.
\begin{figure}[H] 
    \centering
    \includegraphics[width=0.5\textwidth]{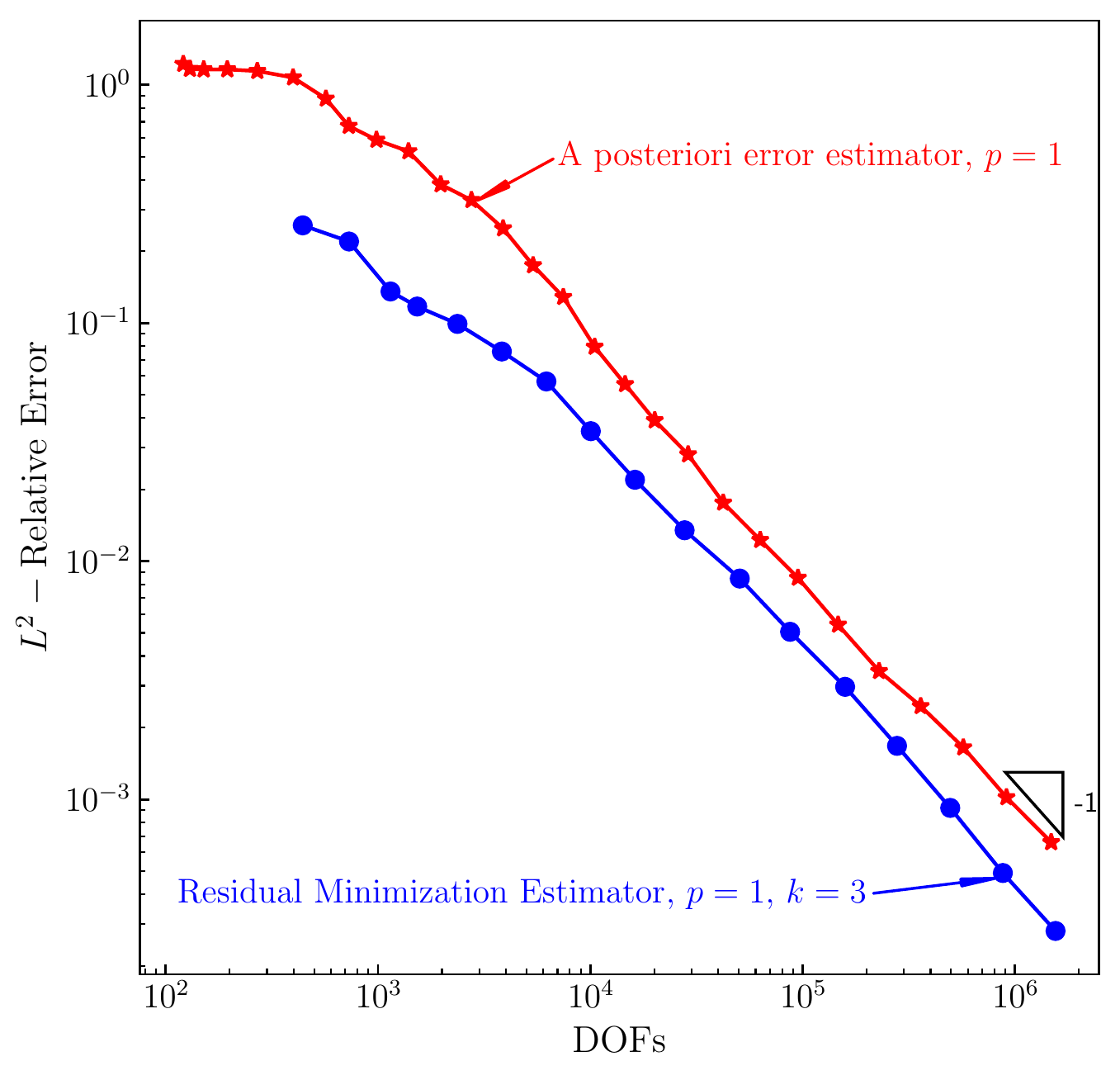}
    \caption{$L^2$-relative error $\left(\frac{\|u-u_h\|_{\Om}}{\|u\|_{\Om}}\right)$.  A-posteriori estimator of~\cite{ burman2009errorCIP} and residual-based estimate.}
    \label{fig2:comparison_min_res_burman}
\end{figure}
Next, we study the validity of the saturation assumption (see Assumption~\ref{assmptn: saturation_min_res}) by considering the triple norm $\tnorm{\cdot}_{h,k}$ from~\eqref{eq:coercive_norm} and introduce the following ratio,
\begin{equation}\label{eq:saturation_ratio}
    \mathcal{S} := \dfrac{\tnorm{u-\theta_h}_{h, k}}{\tnorm{u-u_h}_{h,k}}.
\end{equation}
When $\mathcal{S}<1$, the approximation satisfies the saturation assumption~\ref{assmptn: saturation_min_res}. Figure~\ref{fig4:robustness_saturation} shows the residual representative's robustness and saturation assumption (i.e., Assumption~\ref{assmptn: saturation_min_res}) using piecewise continuous trial space of order $p=1$ and $p=2$ and test space as its enrichment with bubbles of order $k=p+2$. Figure~\ref{fig_4a:robustness} shows that $\tnorm{\varepsilon_h}_{h,k}$ is robust as it is a lower bound of $\tnorm{u - u_h}_{h,k}$ (see Proposition~\ref{prpstn:robustness_residual_repr}). In addition, the saturation Assumption~\ref{assmptn: saturation_min_res} also holds even in the pre-asymptotic regime (see Figure~\ref{fig_4b:saturation}).
\begin{figure}[H] 
    \centering
    \begin{subfigure}[b]{0.44\textwidth}
        \centering
        \includegraphics[width=\textwidth]{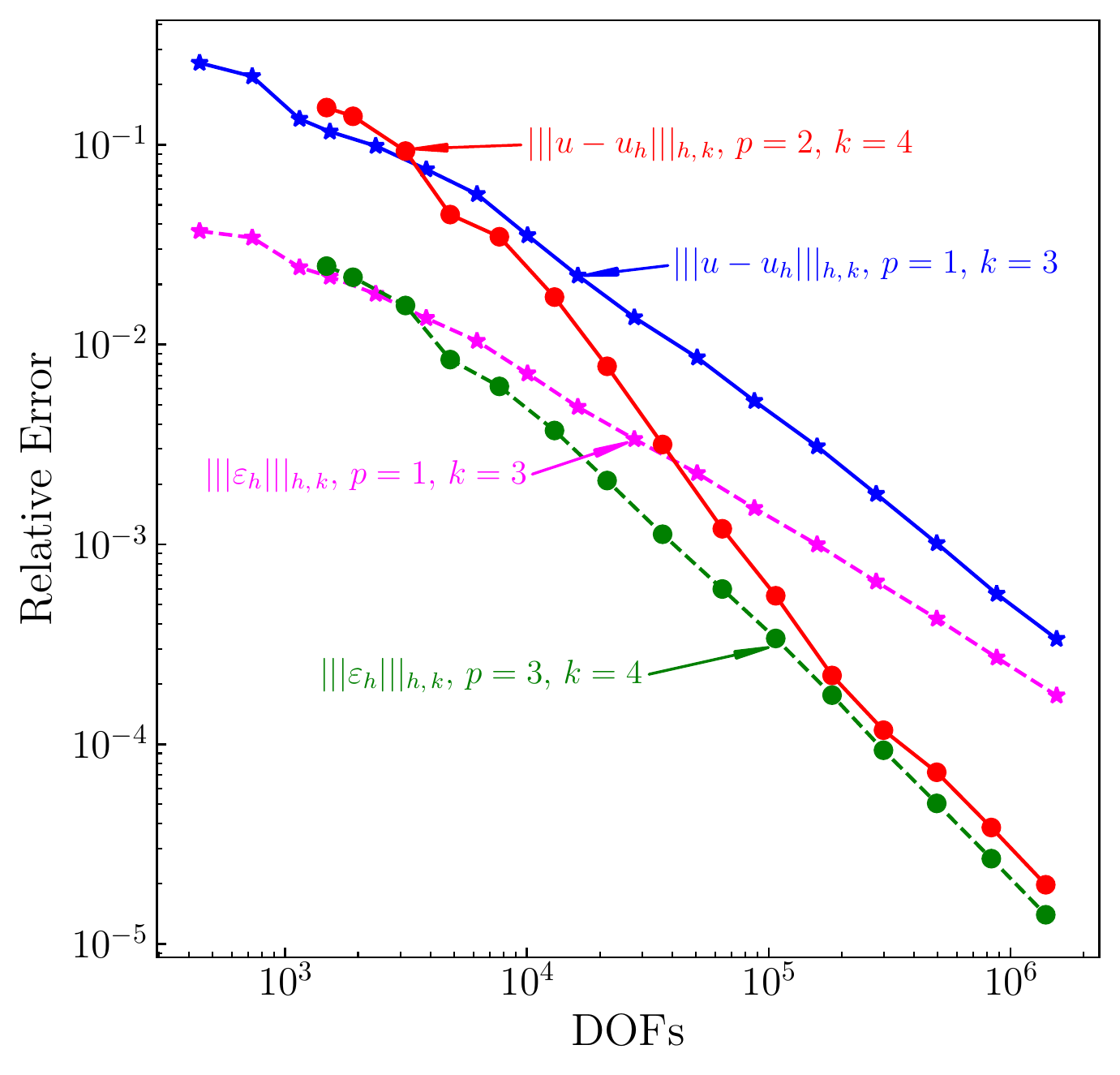}
        \caption{Residual estimator robustness}
        \label{fig_4a:robustness}
    \end{subfigure}
    \hspace{0.2cm}
    \begin{subfigure}[b]{0.45\textwidth} 	 
        \centering
        \includegraphics[width=\textwidth]{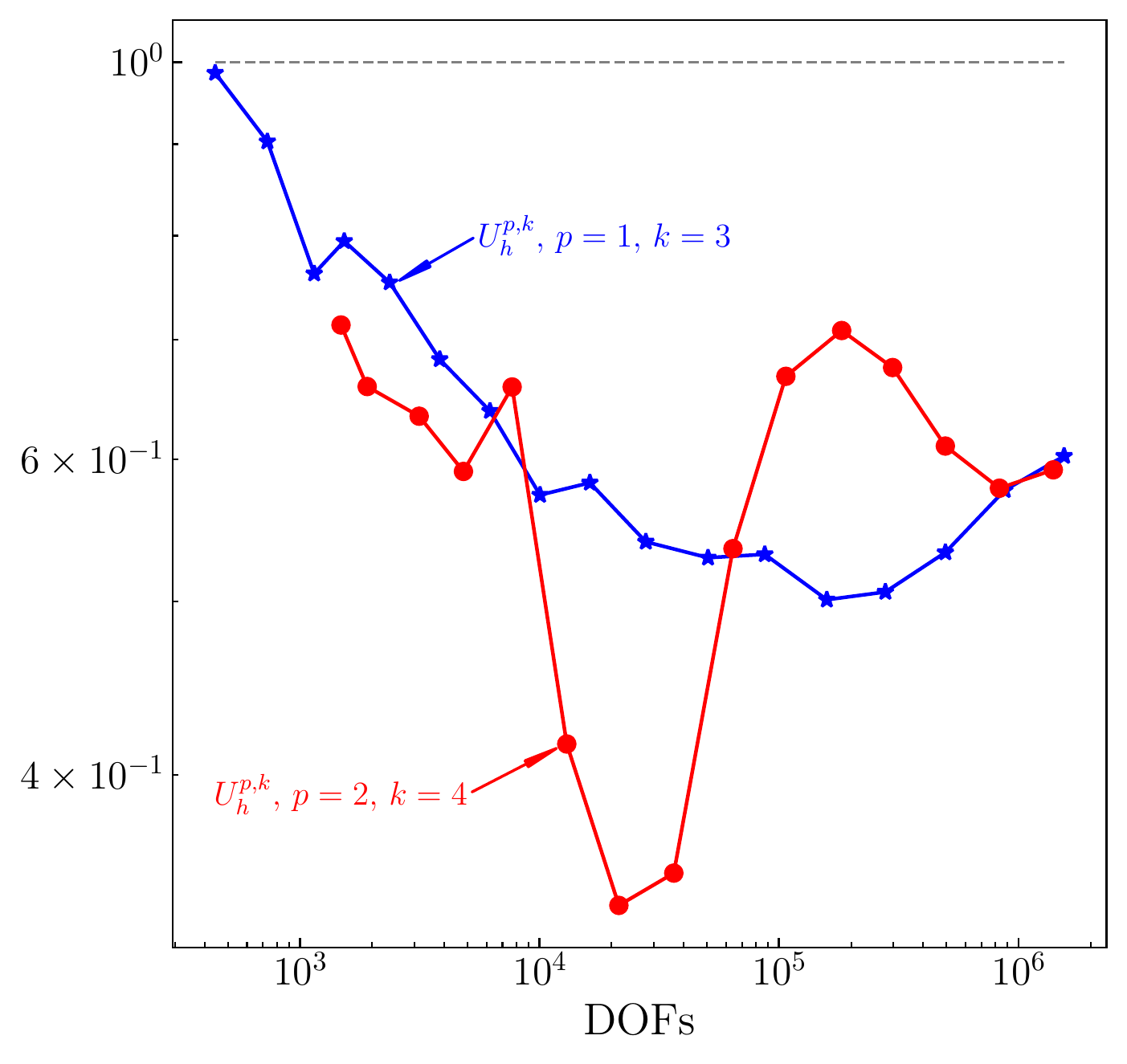}
        \caption{Saturation}
        \label{fig_4b:saturation}
    \end{subfigure}
    \caption{Relative error in the triple norm $\tnorm{\cdot}_{h,k}$ \& saturation assumption.}
    \label{fig4:robustness_saturation}
\end{figure}

\subsection{Goal-oriented adaptivity for advection-reaction problems}

Following~\cite{ rojas2021GoA}, we perform Goal-Oriented Adaptivity (GoA). The key insight behind this theory is to consider an adjoint continuous formulation and an adjoint problem to the saddle-point formulation~\eqref{eq:residual}. Thus, we approximate a quantity of interest (QoI) $q(u)$, where $q: U \mapsto \R$ is a bounded linear functional and $u \in U$ is the exact solution of the continuous problem~\eqref{eq:model_problem}.
The GoA strategy of~\cite{ rojas2021GoA} solves an additional continuous adjoint problem, which is equivalent to solving the following saddle-point problem as the adjoint formulation of problem~\eqref{eq:residual}:
\begin{equation}\label{eq:mixed_adj}
  \left\{ 
    \begin{array}{l}
      \text{Find } (\nu^*_h, w^*_h) \in U^{p,k}_h \times U^p_h, \text{ such that:}\\ 
      \begin{array}{lll}
        (\nu^*_h, \, \nu_h)_{U^{p,k}_h} + b_h(w^*_h, \, \nu_h) \hspace{-0.2cm} & = \ 0,   &\quad \forall\, \nu_h \in U^{p,k}_h, \\
        b_h(w_h, \, \nu^*_h) & = \  q(w_h),  &\quad \forall\, w_h \in U^p_h,
      \end{array}
    \end{array}
  \right.
\end{equation}
where $\nu^*_h \in U^{p,k}_h$ approximates the discrete solution of the adjoint continuous problem (see~\cite{ rojas2021GoA}); the adjoint counterpart of the solution $u_h \in U^p_h$ of~\eqref{eq:residual} while $w^*_h \in U^p_h$ is an additional variable that constrains the solution dimension. Moreover, as the direct saddle-point problem~\eqref{eq:residual} is well-posed, the adjoint saddle-point problem~\eqref{eq:mixed_adj} is also well-posed; the two problems share the same left-hand side. Thus, we approximate the error in the QoI following~\cite{ rojas2021GoA}, which solves an auxiliary discrete problem that has a unique solution and is well-posed,
\begin{equation}\label{eq:res_representation_problem}
\left\{\begin{array}{l}
	\text{Find } \varepsilon^*_h \in U^{p,k}_h \,  \text{ such that:}\\
	 (\varepsilon^*_h, \nu_h) = q(\nu_h) - b_h(\nu_h, \nu^*_h), \quad \forall \, \nu_h \in U^{p,k}_h
\end{array}
\right.
\end{equation}
We adopt the following adjoint-residual-based estimator for the QoI~\cite{ rojas2021GoA},
\begin{equation}\label{eq:adjoint_based_estimator}
    \mathbb{E}^2_T(\varepsilon_h, \varepsilon^*_h):= \big|\big(\varepsilon_h, \varepsilon^*_h \big)_{ U^{p,k}_h} \big|
\end{equation}
which solves~\eqref{eq:residual} and~\eqref{eq:mixed_adj} along with the adjoint residual problem~\eqref{eq:res_representation_problem}. Then, we mark each element with its local upper bound $\tnorm{\varepsilon_h}_{T_{h,k}}\tnorm{\varepsilon^*_h}_{T_{h,k}}$ where $\tnorm{\cdot}_{T_{h,k}}$ is the element localized triple norm. \citet{ rojas2021GoA} built the theory using a dG framework; nevertheless, it is general and can be applied to the bubble-enriched continuous test spaces as well. In our numerical example, we consider the advection-reaction problem~\eqref{eq:model_problem} over the unit square domain $\Om = [0,\,1]^2\in \mathbb{R}^2$, with a constant velocity field $b = (3, 1)^T$ (cf.~\cite{ rojas2021GoA}). Let the reaction parameter be $\mu=0$ (advection dominated), thus the  source term is $f=\mu\,u=0$ in $\Om$. In this case, $g(x_1, x_2)$ defines the exact solution as,
\begin{equation}\label{eq:exact_solution_goa}
    u(x_1, x_2) = 2 + \tanh\left(10\left(x_2 - \dfrac{x_1}{3}-\dfrac{1}{4}\right)\right) + \tanh\left(1000\left(x_2-\dfrac{x_1}{3}-\dfrac{3}{4}\right)\right).
\end{equation}
\begin{figure}[H] 
    \centering
    \begin{subfigure}[b]{0.3\textwidth}
        \centering
        \includegraphics[width=\textwidth]{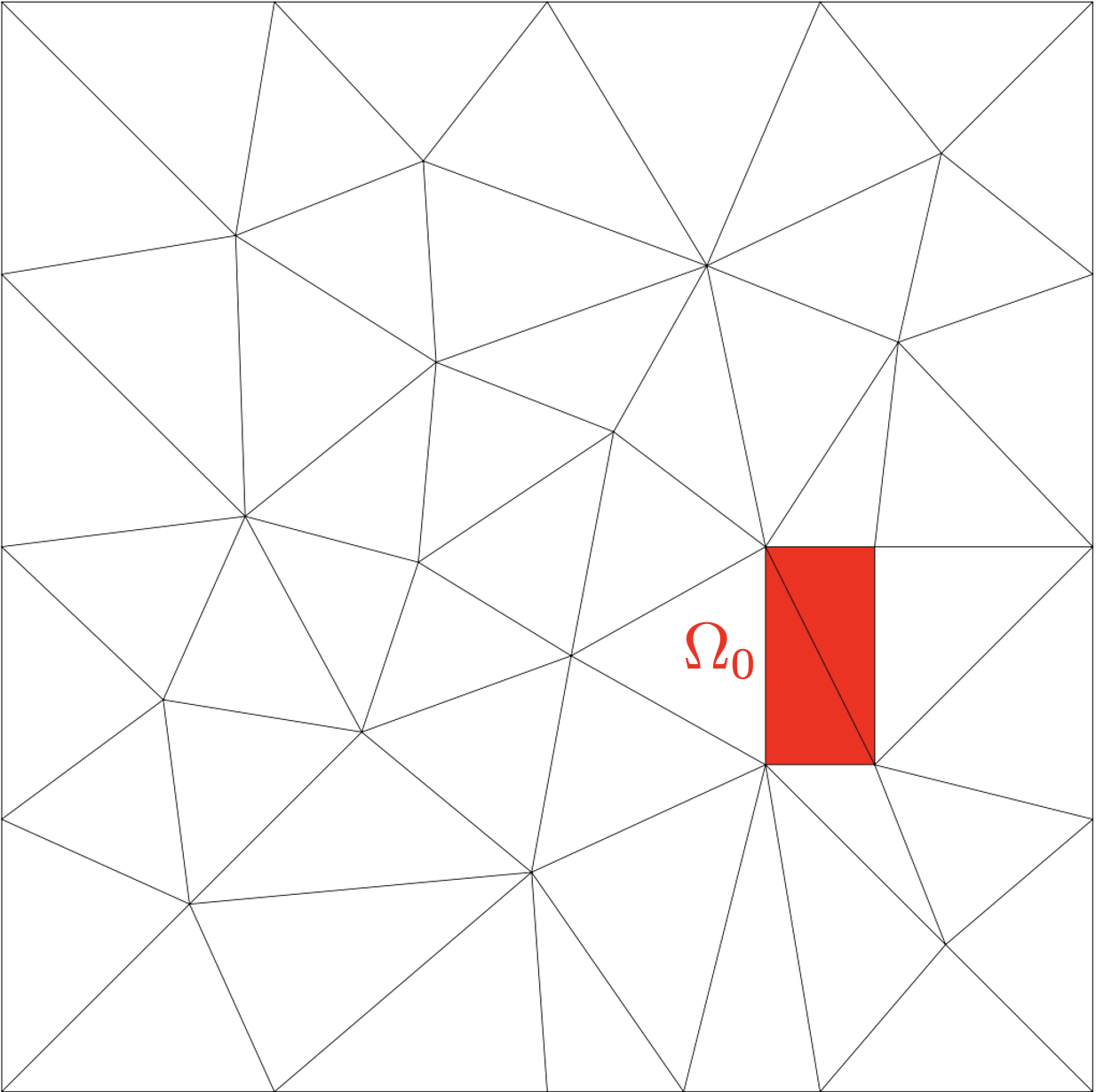}
        \caption{Initial mesh}
        \label{fig_6a:initial_mesh}
    \end{subfigure}
    \hspace{0.2cm}
    \begin{subfigure}[b]{0.3\textwidth} 	 
        \centering
        \includegraphics[width=\textwidth]{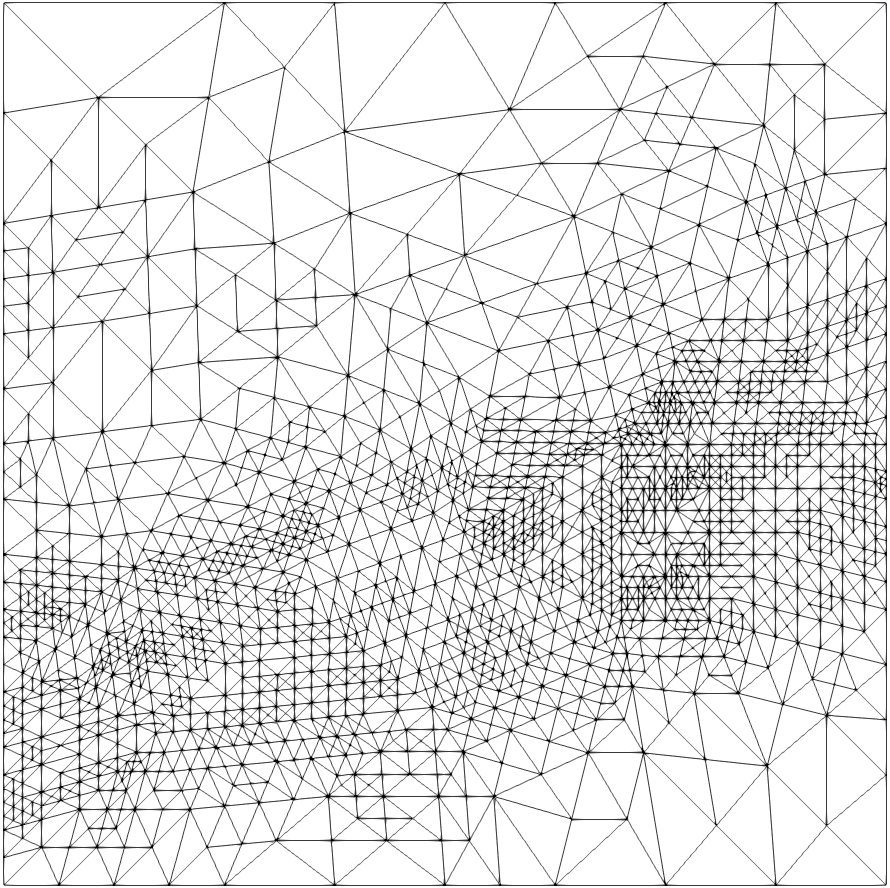}
        \caption{Bubble-enriched CIP (110,990 $\textrm{DoF}$)}
        \label{fig_6b:refined_mesh_CIP}
    \end{subfigure}
    \hspace{0.2cm}
    \begin{subfigure}[b]{0.3\textwidth} 	 
        \centering
        \includegraphics[width=\textwidth]{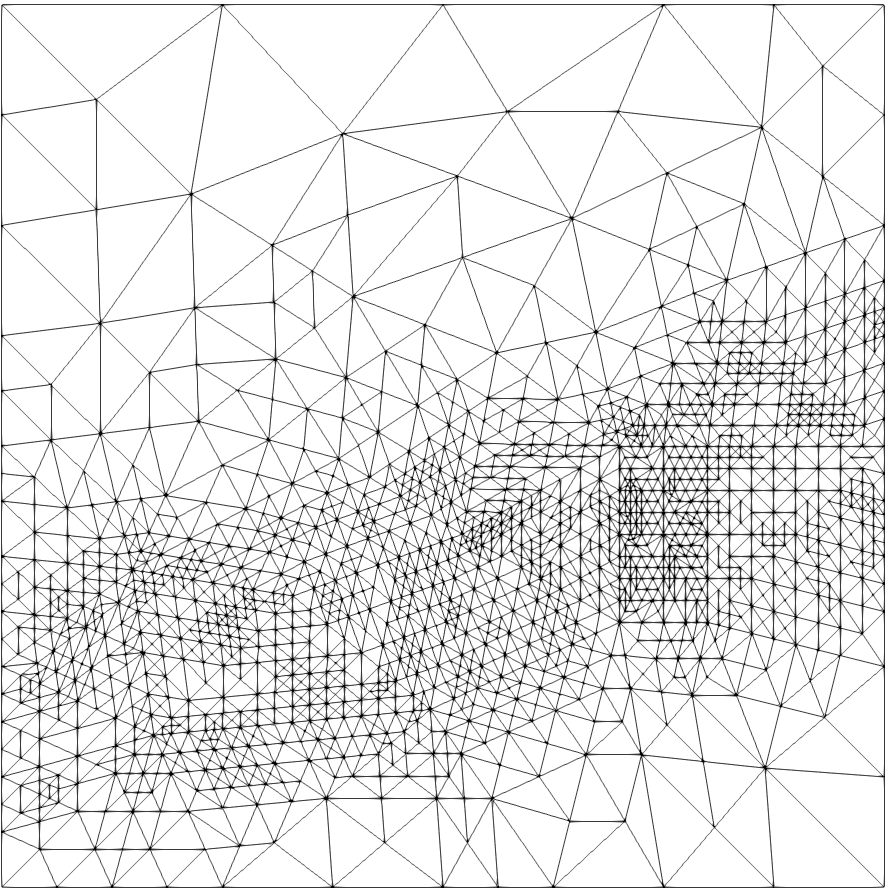}
        \caption{dG framework (107,040 $\textrm{DoF}$)}
        \label{fig_6c:refined_mesh_dg}
    \end{subfigure}
    \caption{Initial \& refined meshes for similar total degrees-of-freedom $\textrm{DoF}$ number using CIP formulation with a bubble enriched test space $(p=1,\, k=3)$ \& a dG framework $(p=1,\, \Delta p=0)$}
    \label{fig6:refined_meshes_num_ex_2}
\end{figure}
\begin{figure}[H] 
    \centering
    \includegraphics[width=0.5\textwidth]{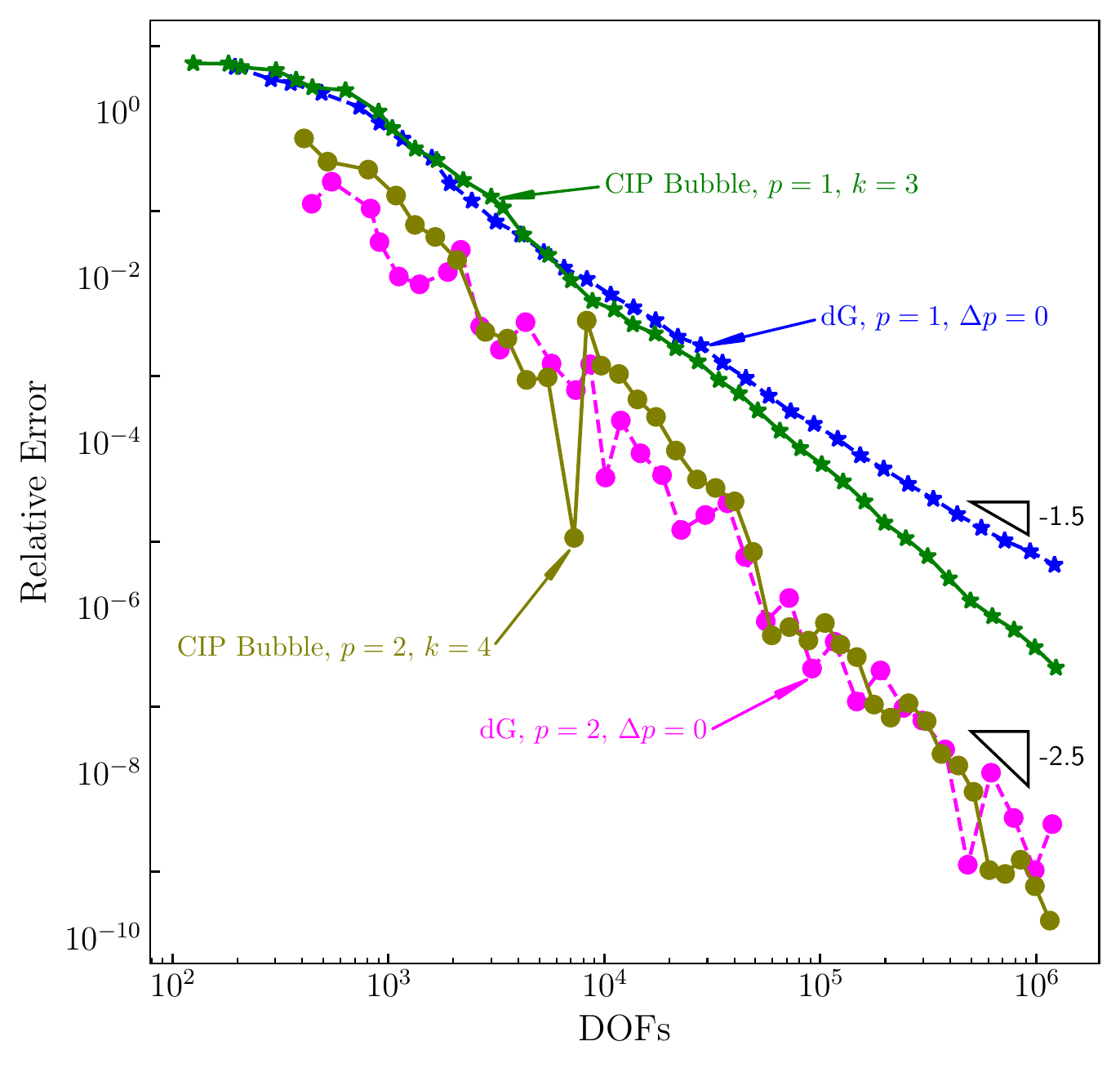}
    \caption{Relative error $(|q(u-u_h)|/|q(u)|$ in the quantity of interest (QoI) using residual minimization with a CIP formulation with bubble-enriched test space (solid lines) and a dG framework (dashed lines)}
    \label{fig5:comparison_goa_CIP_dG}
\end{figure}
In addition, the QoI has the following form,
\begin{equation}
    q(u) = \dfrac{1}{|\Om_0|} \int_{\Om_0} u d\Om_0,
\end{equation}
where $u$ is the exact solution~\eqref{eq:exact_solution_goa} and $\Om_0 = (0.7, 0.8) \times (0.3, 0.5)$ is a subdomain of the physical domain $\Om$. As the starting point for the adaptive procedure, we consider the $\Om_0$-conforming mesh of Figure~\ref{fig_6a:initial_mesh}. 
Figure~\ref{fig6:refined_meshes_num_ex_2} displays the resulting adapted meshes using a CIP formulation with bubble-enriched test space and a dG framework at a similar total number of $\textrm{DoFs}$. These figures show that the resulting adapted meshes are similar, and both methodologies adjust the refinement process consistently to the physical problem nature.
Next, we compare error plots against the expected optimal convergence slope $\textrm{DoF}^{-\left(p+\frac{1}{2}\right)}$ where $p$ is the polynomial order of the trial space (see~\cite{ feischl2016abstractAnalGoA}). 
Since the error is measured in a QoI, we compare the results obtained in~\cite{ rojas2021GoA} using a dG framework against CIP method with residual minimization onto a bubble-enriched test space results directly. We evaluate the numerical results considering two polynomial orders for the trial space (namely $p=1,\,2$) with the same polynomial order for the dG test space (i.e., $\Delta p = 0$) and $k=p+2$ for bubble enriched test space. Figure~\ref{fig5:comparison_goa_CIP_dG} shows the convergence of the relative error in the QoI for the CIP formulation with residual minimization onto bubble-enriched test spaces (solid lines) and the dG framework proposed in~\cite{ rojas2021GoA} (dashed lines). For low order approximations, the CIP formulation with residual minimization shows super convergence compared to the residual minimization using a dG test space that achieves the optimal convergence rate. For higher-order polynomial spaces, CIP has a performance similar to the dG framework.

\section{Contributions and future work.}\label{sec:conclusions}
This paper extends AS-FEM to use the $hp$-CIP finite element method with an enriched test space. We use a stable formulation (namely, the CIP formulation) in the trial space and enlarge the test space via bubble enrichment of the trial space to estimate a residual representative to guide adaptivity, see~\cite{ Labanda2022}. This is meaningful as, in this case, the formulation is already stable in the trial space. Since we choose the coercive stable CIP formulation, we derive a new a-priori error estimate result for the residual minimization method using continuous bubble-enriched test spaces, which relies on an orthogonality argument for the boundedness of the discrete bilinear form that proves quasi-optimal convergence for advection-reaction problems. We also confirm numerically that the residual estimate is robust regarding the energy norm and competitive with the a-posteriori error estimate available in the literature.
Moreover, we show that using the CIP formulation and residual minimization onto bubble-enriched test spaces improves the convergence rates in the $L^2$-norm. However, the residual estimator is calculated using the energy norm. The numerical results for goal-oriented adaptivity suggest that applying residual minimization and a coercive stable formulation using bubble-enriched continuous test spaces improves convergence rates for the error in the quantity of interest when the polynomial order for the trial space is low. Additionally, the method maintains the optimal convergence rates for high-order spaces.
Further research will explore the performance of this method applied to other challenging problems. For example, we will study the extension of this method to non-linear hyperbolic equations that describe many engineering problems, such as solving the shallow water equations and coupled fluid and solid mechanics equations for assessing the long-term stability of tailing dams.

\section*{Acknowledments}
\noindent This publication was made possible in part by the support of Vista Energy Company, which promoted the exploration of new numerical techniques for simulating industry-related problems. Additionally, this publication was made possible in part by the Professorial Chair in Computational Geoscience at Curtin University. This project has received funding from the European Union's Horizon 2020 research and innovation programme under the Marie Sklodowska-Curie grant agreement No 777778 (MATHROCKS).

\appendix{}
\section{Proof of Proposition~\ref{prop:bubble}.}\label{Appendix}

We separate the proof of Proposition~\ref{prop:bubble} into two part: Lemma~\ref{lmm:boundednessPp} and Theorem~\ref{thrm:Aposteriori}, with the a-priori error estimate onto bubble enrichment.

\begin{lmm}[Boundedness]\label{lmm:boundednessPp} Consider $k=p+1>d=2$. If $~\eta \in \left(U^{p}_h \right)^\perp :=\left\{z \in L^2(\Omega): (z, \nu_h) = 0, \, \forall \nu_h \in U^{p}_h\right\}$ the following result holds,
\begin{equation}\label{eq:sup_b_pp1}
    \underset{\nu_h \in U^{p,k}_h}{\sup} \dfrac{b_h(\eta, \nu_h)}{\tnorm{\nu_h}_{k, h}} \lesssim \tnorm{\eta}_{k, h,\#},
\end{equation}
where the norm 
$$\tnorm{\eta}_{k, h,\#} := \tnorm{\eta}_{k, h} + \big(k^2 \, h^{\frac{1}{2}}p^{1/2} + k^{\alpha/2}p^{-1/2} \big)|\eta|_h $$ 
%
%
and  $|\cdot|_h,$ is the semi-norm
$$|\eta|_h :=\left(\sum_{T\in \Omega_h} h^{-1}_T\beta_{T,\infty} \|\eta\|^2_T\right)^{1/2} .$$

\end{lmm}

\begin{proof} 
    The only term to estimate in~\eqref{eq:CIP_ discrete_bubbles}  is $\big(\eta,\,\Vb \cdot \nabla \nu_h \big)_\Om$. 
    Let $\Vb_h$ the $L^2$-orthogonal projection of $\Vb$ onto $V_h^0$. Thus,
    \begin{equation}\label{eq:advective_term}
        \big(\eta, \Vb \cdot \nabla \nu_h \big)_\Om = \big(\eta, (\Vb -\Vb_h) \cdot 	
        \nabla \nu_h \big)_\Om + \big(\eta, \Vb_h\nabla\nu_h \big)_\Om.
    \end{equation}
    Since  $\Vb \in \big[W^{1, \infty}(\Om) \big]^2$, and $\nu_h\in U_h^{p,k}\subset U_h^{k}$ we have the following inverse inequalities (see e.g.,~\cite{ canuto1982approximation}): 
    \begin{align}\|\Vb-\Vb_h\|_{[L^\infty(T)]^d}&\lesssim h_T\|\Vb\|_{[W^{1,
        \infty}(T)]^d}, \label{eq:inverseW1}\\
        \|\nabla v_h\|_{T}&\lesssim k^2h^{-1}_T\|v_h\|_{T}, \text{ for all} T \in \Omega_h. \label{eq:inverseL2}
    \end{align}
    Thus, we can bound the first term in the right-hand side of~\eqref{eq:advective_term} as follows:
    \begin{align*}
        \big(\eta, (\Vb - \Vb_h) \cdot \nabla \nu_h \big)_\Om &\lesssim \left(  
        \sum_{T 	\in \Om_h} h_T^{-1}  \Gb_{T, \infty} \norm{\eta}^2_T \right)^{\frac{1}
        {2}} \left( \sum_{T \in \Om_h} h_T^3 \Gb_{T, \infty}^{-1} \norm{\Vb}_{W^{1,
        \infty}(T)}^2 \norm{\nabla \nu_h}_T^2 \right)^{\frac{1}{2}} \quad \text{(by~\eqref{eq:inverseW1} and Cauchy-Schwartz)} \nonumber\\
        & \lesssim k^2 h^{\frac{1}{2}} |\eta|_h \norm{\nu_h}_\Om \hspace{6.13cm} \
        \text{(by~\eqref{eq:semi-norm}, and~\eqref{eq:inverseL2})}.
    \end{align*}
    The second term in the right-hand side of~\eqref{eq:advective_term} is estimated using the fact that $\eta\in \left(U^{p}_h\right)^\perp$, the Cauchy-Schwartz inequality and the interpolation estimate of Lemma~\ref{hp_interp_error}.
    Let $ \phi_h:=\Vb_h \cdot \nabla \nu_h$. We call $U_{\beta,h}:=\left\{I_{Os}(\phi_h):\phi_h =\Vb_h \cdot \nabla \nu_h, \quad \nu_h\in U^{k,p}_h \right\}.$
    Since, 
    $I_{Os}\phi_h$  may not be in $U^{p,k}_h$ but it is contained in $U^{p}_h$, since $k=p+1$. Hence, 
    \begin{align*}
        \big( \eta, \phi_h)_\Om & = \big( \eta,\phi_h - I_{Os}\phi_h\big)_\Om \\
        & \leq \left(  \sum_{T \in \Om_h} h_T^{-1}  \Gb_{T, \infty} \norm{\eta}^2_T \right)^{\frac{1}{2}} \left( \sum_{T \in \Om_h} h_T \Gb_{T, \infty}^{-1} \norm{\phi-I_{Os}\phi_h}^2_T\right)^{\frac{1}{2}}  \\
        & \lesssim |\eta|_h \left( \sum_{T\in \Om_h} \sum_{e \in \mathcal{F}_T} h_T^2 p^{-1} \Gb_{T, \infty}^{-1} \norm {\llbracket \phi_h \rrbracket}^2_e \right)^{\frac{1}{2}}. 
    \end{align*}
    Since 
    \begin{align}
        \norm {\llbracket \phi_h \rrbracket}_e =\norm {\llbracket \Vb_h\cdot \nabla \nu_h \rrbracket}_e \lesssim  \norm {\llbracket(\Vb- \Vb_h)\cdot \nabla \nu_h \rrbracket}_e + \norm{\llbracket \Vb\cdot \nabla \nu_h \rrbracket}_e,
    \end{align}
    and
    \begin{align}
        \norm {\llbracket(\Vb- \Vb_h)\cdot \nabla \nu_h \rrbracket}_e 
        & \lesssim \sum_{e\in \overline{T} }h_T\|\Vb\|_{[W^{1,\infty}(T)]^d} \|\nabla \nu_h|_T\|_e \\
        & \lesssim  \sum_{e\in \overline{T} }h_T\|\Vb\|_{[W^{1,\infty}(T)]^d} \left(\frac{p^2}{h_T}\right)^{1/2}\|\nabla \nu_h\|_T & \text{( Estimate~\eqref{eq:trace-inequality})}\\
        & \lesssim  \sum_{e\in \overline{T} }h_T\|\Vb\|_{[W^{1,\infty}(T)]^d} \left(\frac{p^2}{h_T}\right)^{1/2}\frac{k^2}{h_T}\| \nu_h\|_T & \text{(Estimate~\eqref{eq:inverseL2})}\\
        & \lesssim  \sum_{e\in \overline{T} }\|\Vb\|_{[W^{1,\infty}(T)]^d}\frac{pk^2}{h_T^{1/2}}\| \nu_h\|_T. 
    \end{align}
    Hence 
    \begin{align}
        \sum_{T\in \Om_h} \sum_{e \in \mathcal{F}_T} h_T^2 p^{-1} \Gb_{T, \infty}^{-1} \norm {\llbracket (\Vb- \Vb_h)\cdot \nabla v_h \rrbracket}^2_e \lesssim {p}k^4h \|v_h\|_\Omega^2\nonumber
    \end{align}
    Furthermore, 
    \begin{align}
        \sum_{T\in \Om_h} \sum_{e \in \mathcal{F}_T} h_T^2 p^{-1} \Gb_{T, \infty}^{-1} \norm {\llbracket \Vb\cdot \nabla \nu_h \rrbracket}^2_e & \lesssim \sum_{T\in \Om_h} \sum_{e \in \mathcal{F}_T} h_e^2 p^{-1} \Gb_{e, \infty}\norm {\llbracket \nabla v_h \cdot n \rrbracket}_e^2 \\
        & \lesssim p^{-1}k^{\alpha}j(v_h, v_h).
        %
    \end{align}
    Collecting the above estimates and Lemma~\ref{lmm:Stability_bubbles}  yields
    \[
    \big(\eta, \phi_h \big)_\Om \lesssim |\eta|_h \big(k^2 \, h^{\frac{1}{2}}p^{1/2} + k^{\alpha/2}p^{-1/2}  \big) \tnorm{\nu_h}_{k, h}.
    \]

    %
    \noindent This completes the proof.
    \end{proof}

%


\begin{lmm}\label{Lemma12}
    Consider the continuous $L^2(\Omega)$-projection operator $\Pi_h:L^2(\Omega)\to U_h^{p}$. Then, for all $w\in H^s(\Omega)$, $s\geq 1$, we have the following $hp$-approximation properties
    \begin{align}
        | w - \Pi_h w |_{h} & \lesssim  p^{-\frac{1}{4}}\left(\frac{h}{p}\right)^{q-1/2}\|w\|_{q,\Omega}, \label{eq:estimateH}\\
        \tnorm{w-\Pi_h w}_{h,k} &\lesssim \left(p^{3/4}+p^{11/4}k^{-\alpha/2}\right)\left(\frac{h}{p}\right) ^{q-1/2}\|w\|_{q,\Omega} \label{eq:estimateV}
    \end{align}
    with $q = \min\{p+1,s\}$.
\end{lmm}

\begin{proof} 

Let $w \in H^s(\Omega), s \geq 1.$ The proof of~\eqref{eq:estimateH} follows from~\cite[Lemma 5.6]{burmanErn2007hpCIP}.

    The proof of estimate~\eqref{eq:estimateV} is as follows. By definition,
    \begin{align}
        \tnorm{w-\Pi_h w}_{h,k}^2 := \big\| \mu_0^{1/2} (w-\Pi_h w) \big\|^2_\Om + \dfrac{1}{2}\big\| |\Vb \cdot \Vn|^{1/2} (w-\Pi_h w) \big\|^2_{\bOm}  + j_{h,k}( w-\Pi_h w, w-\Pi_h w). \label{eq:boundHJ}
    \end{align}
    The estimates~\eqref{eq:estimateL2} and~\eqref{eq:estimateL2grad}  of Lemma~\ref{hp_error_estimate_L2} imply
    \begin{align*}
        \|w-\Pi_hw\|_{\partial\Omega}\leq p^{3/4}\left(\frac{h}{p}\right)^{q-1/2}\|w\|_{q,\Omega}.
    \end{align*}
    Thus, the second term of ~\eqref{eq:boundHJ} is bounded.
    To bound the last term of~\eqref{eq:boundHJ}, we will use the well-known estimate for the orthogonal $L^2$ projection $\Pi^*_h:L^2(\Omega)\to V^p_h$ under the same conditions of~\cite[Lemma 5.4]{ burmanErn2007hpCIP},
    \begin{align}
        \|w - \Pi^*_hw\|_{\partial T}&\lesssim p^{1/4}\left(\frac{h}{p}\right)^{q-1/2}\|w\|_{q,T}, \qquad \forall w \in \mathbb{P}^{p}(T)\label{eq:estimateE}
    \end{align}
    and 
    \begin{align*}
        \|\nabla(w-\Pi_h w)\|_{\partial T } & \leq \|\nabla(w-\Pi^*_h w)\|_{\partial T } +\|\nabla(\Pi_h w-\Pi^*_h w)\|_{\partial T } \\
        & \lesssim \|\nabla (w-\Pi^*_hw)\|_{\partial T} +\left(\frac{p^2}{h_T}\right)^{3/2}\|\Pi_h w - \Pi^*_h w\|_T & \text{(estimates~\eqref{eq:trace-inequality} and~\eqref{eq:inverseL2})}\\
        &\lesssim \|\nabla (w-\Pi^*_hw)\|_{\partial T} +\left(\frac{p^2}{h_T}\right)^{3/2}\|\Pi_h^* w - I_{Os}(\Pi^*_h w)\|_T\\ 
        &\lesssim \|\nabla (w-\Pi^*_hw)\|_{\partial T} +\left(\frac{p^2}{h_T}\right)^{3/2}\left(\frac{h_T}{p}\right)^{1/2} \sum_{e\in \mathcal{F}_T} \|[\Pi_h^* w- w]_e\|_e\\
        &  \lesssim \|\nabla (w-\Pi^*_hw)\|_{\partial T} +\left(\frac{p^2}{h_T}\right)^{3/2}\left(\frac{h_T}{p}\right)^{1/2}p^{1/4}\left(\frac{h_T}{p}\right)^{q-1/2}\|w\|_{q,T} & \text{(estimate~\eqref{eq:estimateE}}\\
        & \lesssim  \|\nabla (w-\Pi^*_hw)\|_{\partial T} +{p}^{7/4}\left(\frac{h_T}{p}\right)^{q-3/2}\|w\|_{q,T} & \\
        & \lesssim p^{7/4}\left(\frac{h}{p}\right)^{q-3/2}\|w\|_{q,T}.
    \end{align*}
    Thus, the last term of~\eqref{eq:boundHJ} is controlled as follows
    \begin{align*}
        j_{h,k}(w-\Pi_h w, w-\Pi_h w) & \lesssim \sum_{e\in \mathcal{F}}\frac{h_e^2}{k^\alpha}\beta_{\infty,e}\|[\nabla(w-\Pi w)\cdot n]_e\|_e^2 \\
        & \lesssim \sum_{T\in \Omega_h}\sum_{e\in \mathcal{F}(T)}\frac{h^2_T}{k^\alpha}\beta_{\infty,e}\left(p^{7/2}\left(\frac{h_T}{p}\right)^{2q-3}\|w\|_{q,\Omega}^2\right)\\
        & \lesssim p^{11/4}k^{-\alpha}\left(\frac{h}{p}\right)^{2q-1}\|w\|_{q,\Omega}^2.
    \end{align*}
\end{proof}



 


\begin{thrm}[A priori error estimate onto bubble enrichment]\label{thrm:Aposteriori}
	Let $u\in H^s(\Om)$ with $s > \frac{3}{2}$, solve 
	\eqref{eq:model_problem} and let $u_h\in U_h^{p}$ solve~\eqref{eq:residual}. Take $\alpha=\frac{7}{2}$, then  under Assumption~\ref{ass:reference} and using Lemma~\ref{Lemma12},   the following estimate holds
	\begin{equation}
		\tnorm{u - u_h}_{k, h} \lesssim (p + p^{\frac{5}{4}} h^{\frac{1}{2}}) \left(\dfrac{h}{p} \right)^{q-\frac{1}{2}} \norm{u}_{q, \Om},\label{eq:estimate1}
	\end{equation}
	with $q=\min\{p+1, s\}$. Moreover, if $~h \leq p^{-\frac{1}{2}}$ then,
	\begin{equation}
		\tnorm{u - u_h}_{k, h} \lesssim p \left(\dfrac{h}{p} \right)^{q-\frac{1}{2}} \norm{u}_{q, \Om}.
        \end{equation}
\end{thrm}
\begin{proof}
When $\alpha=7/2,$ following estimate ~\eqref{eq:estimateV} in Lemma~\ref{Lemma12}, we have 
\begin{align*}
		\tnorm{u - \Pi_h u }_{k, h} & \lesssim \left(p+p^{5/4}h^{1/2}\right)\left(\frac{h}{p}\right) ^{q-1/2}\|u\|_{q,\Omega}.
\end{align*}
Moreover, following the proof of Theorem~\ref{thrm:a-priori-err-estim-min-res} we have
\begin{align*}
    \tnorm{u_h-\Pi_h u}_{h,k} &\lesssim \tnorm{u-\Pi_h u}_{h,k,\#}
\end{align*}

Thus, 
   \begin{align*}
		\tnorm{u - u_h}_{k, h} & \lesssim \tnorm{u-\Pi_h u}_{h,k} +  \tnorm{\Pi_h u-u_h}_{k,h}\\
  & \lesssim \tnorm{u-\Pi_h u}_{h,k} + \big(k^2 \, h^{\frac{1}{2}}p^{1/2}
  k^{\alpha/2}p^{-1/2}  \big)\big|u-\Pi_h u\big|_h\\
    & \lesssim \tnorm{u-\Pi_h u}_{h,k} +  
\big(k^2 \, h^{\frac{1}{2}} p^{1/2} + k^{7/4}p^{-1/2}  \big)p^{-\frac{1}{4}}\left(\frac{h}{p}\right)^{q-1/2} \|u\|_{q,\Omega} & \text{(using estimate ~\eqref{eq:estimateH})}\\
 & \lesssim  \left(p^{3/4}+p^{11/4}k^{-7/4}+  k^{2}p^{1/4} \, h^{\frac{1}{2}} + k^{7/4}p^{-3/4}  \right)\left(\frac{h}{p}\right)^{q-1/2}
\|u\|_{q,\Omega},& \text{(using estimate~\eqref{eq:estimateV})}\\
& \lesssim\left( p + p^{5/4}h^{\frac{1}{2}}\right)\left(\frac{h}{p}\right)^{q-1/2}
\|u\|_{q,\Omega} & \text{ (overestimating) } 
  \end{align*}
  with $q=\min(p+1,s)$ and the inequality~\eqref{eq:estimate1} is satisfied.
 Moreover,  if $h\leq p^{-{{1}/2}}$,  with $p\geq 1$ can be bounded by 1, we obtain  
 \begin{align*}
     \tnorm{u - u_h}_{k, h}\lesssim p\left(\frac{h}{p}\right)^{q-1/2}
\|u\|_{q,\Omega},
  \end{align*}
  with
$q=\min(p+1,s)$.
\end{proof}
%
\pagebreak
\bibliography{mybibfile}

\begin{thebibliography}{50}
\expandafter\ifx\csname natexlab\endcsname\relax\def\natexlab#1{#1}\fi
\providecommand{\url}[1]{\texttt{#1}}
\providecommand{\href}[2]{#2}
\providecommand{\path}[1]{#1}
\providecommand{\DOIprefix}{doi:}
\providecommand{\ArXivprefix}{arXiv:}
\providecommand{\URLprefix}{URL: }
\providecommand{\Pubmedprefix}{pmid:}
\providecommand{\doi}[1]{\href{http://dx.doi.org/#1}{\path{#1}}}
\providecommand{\Pubmed}[1]{\href{pmid:#1}{\path{#1}}}
\providecommand{\bibinfo}[2]{#2}
\ifx\xfnm\relax \def\xfnm[#1]{\unskip,\space#1}\fi
\bibitem[{Adams \& Fournier(2003)}]{adams2003sobolev}
\bibinfo{author}{Adams, R.~A.}, \& \bibinfo{author}{Fournier, J. J.~F.}
  (\bibinfo{year}{2003}).
\newblock {\it \bibinfo{title}{Sobolev spaces,}\/} volume \bibinfo{volume}{140}
  of {\it \bibinfo{series}{Pure and Applied Mathematics (Amsterdam)}\/}.
\newblock (\bibinfo{edition}{2nd} ed.).
\newblock \bibinfo{publisher}{Elsevier/Academic Press, Amsterdam}.
\bibitem[{Alnaes et~al.(2015)Alnaes, Blechta, Hake, Johansson, Kehlet, Logg,
  Richardson, Ring, Rognes \& Wells}]{fenics2015}
\bibinfo{author}{Alnaes, M.~S.}, \bibinfo{author}{Blechta, J.},
  \bibinfo{author}{Hake, J.}, \bibinfo{author}{Johansson, A.},
  \bibinfo{author}{Kehlet, B.}, \bibinfo{author}{Logg, A.},
  \bibinfo{author}{Richardson, C.}, \bibinfo{author}{Ring, J.},
  \bibinfo{author}{Rognes, M.~E.}, \& \bibinfo{author}{Wells, G.~N.}
  (\bibinfo{year}{2015}).
\newblock \bibinfo{title}{The {FEniCS} project version 1.5}.
\newblock {\it \bibinfo{journal}{Archive of Numerical Software}\/},  {\it
  \bibinfo{volume}{3}\/}. \DOIprefix\doi{10.11588/ans.2015.100.20553}.
\bibitem[{Arnold et~al.(1984)Arnold, Brezzi \& Fortin}]{arnold1996mini}
\bibinfo{author}{Arnold, D.~N.}, \bibinfo{author}{Brezzi, F.}, \&
  \bibinfo{author}{Fortin, M.} (\bibinfo{year}{1984}).
\newblock \bibinfo{title}{A stable finite element for the {S}tokes equations}.
\newblock {\it \bibinfo{journal}{Calcolo}\/},  {\it \bibinfo{volume}{21}\/},
  \bibinfo{pages}{337--344 (1985)}. \DOIprefix\doi{10.1007/BF02576171}.
\bibitem[{Babu\v{s}ka \& Zl\'{a}mal(1973)}]{babushka1973InteriorPenalty}
\bibinfo{author}{Babu\v{s}ka, I.}, \& \bibinfo{author}{Zl\'{a}mal, M.}
  (\bibinfo{year}{1973}).
\newblock \bibinfo{title}{Nonconforming elements in the finite element method
  with penalty}.
\newblock {\it \bibinfo{journal}{SIAM J. Numer. Anal.}\/},  {\it
  \bibinfo{volume}{10}\/}, \bibinfo{pages}{863--875}. \URLprefix
  \url{https://doi.org/10.1137/0710071}. \DOIprefix\doi{10.1137/0710071}.
\bibitem[{Bank et~al.(1983)Bank, Sherman \& Weiser}]{bank1983refinementAlgo}
\bibinfo{author}{Bank, R.~E.}, \bibinfo{author}{Sherman, A.~H.}, \&
  \bibinfo{author}{Weiser, A.} (\bibinfo{year}{1983}).
\newblock \bibinfo{title}{Refinement algorithms and data structures for regular
  local mesh refinement}.
\newblock In {\it \bibinfo{booktitle}{Scientific computing ({M}ontreal, {Q}ue.,
  1982)}\/} IMACS Trans. Sci. Comput., I (pp. \bibinfo{pages}{3--17}).
\newblock \bibinfo{publisher}{IMACS, New Brunswick, NJ}.
\bibitem[{Bazilevs et~al.(2007)Bazilevs, Calo, Cottrell, Hughes, Reali \&
  Scovazzi}]{Bazilevs2007}
\bibinfo{author}{Bazilevs, Y.}, \bibinfo{author}{Calo, V.},
  \bibinfo{author}{Cottrell, J.}, \bibinfo{author}{Hughes, T.},
  \bibinfo{author}{Reali, A.}, \& \bibinfo{author}{Scovazzi, G.}
  (\bibinfo{year}{2007}).
\newblock \bibinfo{title}{Variational multiscale residual-based turbulence
  modeling for large eddy simulation of incompressible flows}.
\newblock {\it \bibinfo{journal}{Computer Methods in Applied Mechanics and
  Engineering}\/},  {\it \bibinfo{volume}{197}\/}, \bibinfo{pages}{173--201}.
  \DOIprefix\doi{10.1016/j.cma.2007.07.016}.
\bibitem[{Brezzi et~al.(1992)Brezzi, Bristeau, Franca, Mallet \&
  Rog\'{e}}]{brezzi1992bstablerel}
\bibinfo{author}{Brezzi, F.}, \bibinfo{author}{Bristeau, M.~O.},
  \bibinfo{author}{Franca, L.~P.}, \bibinfo{author}{Mallet, M.}, \&
  \bibinfo{author}{Rog\'{e}, G.} (\bibinfo{year}{1992}).
\newblock \bibinfo{title}{A relationship between stabilized finite element
  methods and the {G}alerkin method with bubble functions}.
\newblock {\it \bibinfo{journal}{Comput. Methods Appl. Mech. Engrg.}\/},  {\it
  \bibinfo{volume}{96}\/}, \bibinfo{pages}{117--129}.
  \DOIprefix\doi{10.1016/0045-7825(92)90102-P}.
\bibitem[{Brezzi et~al.(1999)Brezzi, Hughes, Marini, Russo \&
  S\"{u}li}]{brezzi1999RFB}
\bibinfo{author}{Brezzi, F.}, \bibinfo{author}{Hughes, T. J.~R.},
  \bibinfo{author}{Marini, L.~D.}, \bibinfo{author}{Russo, A.}, \&
  \bibinfo{author}{S\"{u}li, E.} (\bibinfo{year}{1999}).
\newblock \bibinfo{title}{A priori error analysis of residual-free bubbles for
  advection-diffusion problems}.
\newblock {\it \bibinfo{journal}{SIAM J. Numer. Anal.}\/},  {\it
  \bibinfo{volume}{36}\/}, \bibinfo{pages}{1933--1948}.
  \DOIprefix\doi{10.1137/S0036142998342367}.
\bibitem[{Broersen et~al.(2018)Broersen, Dahmen \&
  Stevenson}]{broersen2018stability}
\bibinfo{author}{Broersen, D.}, \bibinfo{author}{Dahmen, W.}, \&
  \bibinfo{author}{Stevenson, R.~P.} (\bibinfo{year}{2018}).
\newblock \bibinfo{title}{On the stability of {DPG} formulations of transport
  equations}.
\newblock {\it \bibinfo{journal}{Math. Comp.}\/},  {\it
  \bibinfo{volume}{87}\/}, \bibinfo{pages}{1051--1082}. \URLprefix
  \url{https://doi.org/10.1090/mcom/3242}. \DOIprefix\doi{10.1090/mcom/3242}.
\bibitem[{Burman(2009)}]{burman2009errorCIP}
\bibinfo{author}{Burman, E.} (\bibinfo{year}{2009}).
\newblock \bibinfo{title}{A posteriori error estimation for interior penalty
  finite element approximations of the advection-reaction equation}.
\newblock {\it \bibinfo{journal}{SIAM J. Numer. Anal.}\/},  {\it
  \bibinfo{volume}{47}\/}, \bibinfo{pages}{3584--3607}. \URLprefix
  \url{https://doi.org/10.1137/080733899}. \DOIprefix\doi{10.1137/080733899}.
\bibitem[{Burman \& Ern(2005)}]{burmanErn2005hpCIP}
\bibinfo{author}{Burman, E.}, \& \bibinfo{author}{Ern, A.}
  (\bibinfo{year}{2005}).
\newblock \bibinfo{title}{Stabilized {G}alerkin approximation of
  convection-diffusion-reaction equations: discrete maximum principle and
  convergence}.
\newblock {\it \bibinfo{journal}{Math. Comp.}\/},  {\it
  \bibinfo{volume}{74}\/}, \bibinfo{pages}{1637--1652}.
  \DOIprefix\doi{10.1090/S0025-5718-05-01761-8}.
\bibitem[{Burman \& Ern(2007)}]{burmanErn2007hpCIP}
\bibinfo{author}{Burman, E.}, \& \bibinfo{author}{Ern, A.}
  (\bibinfo{year}{2007}).
\newblock \bibinfo{title}{Continuous interior penalty {$hp$}-finite element
  methods for advection and advection-diffusion equations}.
\newblock {\it \bibinfo{journal}{Math. Comp.}\/},  {\it
  \bibinfo{volume}{76}\/}, \bibinfo{pages}{1119--1140}.
  \DOIprefix\doi{10.1090/S0025-5718-07-01951-5}.
\bibitem[{Burman \& Hansbo(2004)}]{burman2004edgebased}
\bibinfo{author}{Burman, E.}, \& \bibinfo{author}{Hansbo, P.}
  (\bibinfo{year}{2004}).
\newblock \bibinfo{title}{Edge stabilization for {G}alerkin approximations of
  convection-diffusion-reaction problems}.
\newblock {\it \bibinfo{journal}{Comput. Methods Appl. Mech. Engrg.}\/},  {\it
  \bibinfo{volume}{193}\/}, \bibinfo{pages}{1437--1453}.
  \DOIprefix\doi{10.1016/j.cma.2003.12.032}.
\bibitem[{Burman \& Hansbo(2006)}]{burman2006edgebased}
\bibinfo{author}{Burman, E.}, \& \bibinfo{author}{Hansbo, P.}
  (\bibinfo{year}{2006}).
\newblock \bibinfo{title}{Edge stabilization for the generalized {S}tokes
  problem: a continuous interior penalty method}.
\newblock {\it \bibinfo{journal}{Comput. Methods Appl. Mech. Engrg.}\/},  {\it
  \bibinfo{volume}{195}\/}, \bibinfo{pages}{2393--2410}.
  \DOIprefix\doi{10.1016/j.cma.2005.05.009}.
\bibitem[{Calo et~al.(2014)Calo, Collier \& Niemi}]{Calo2014}
\bibinfo{author}{Calo, V.~M.}, \bibinfo{author}{Collier, N.~O.}, \&
  \bibinfo{author}{Niemi, A.~H.} (\bibinfo{year}{2014}).
\newblock \bibinfo{title}{Analysis of the discontinuous petrov–galerkin
  method with optimal test functions for the reissner–mindlin plate bending
  model}.
\newblock {\it \bibinfo{journal}{Computers \& Mathematics with
  Applications}\/},  {\it \bibinfo{volume}{66}\/}, \bibinfo{pages}{2570--2586}.
  \DOIprefix\doi{10.1016/j.camwa.2013.07.012}.
\bibitem[{Calo et~al.(2020)Calo, Ern, Muga \& Rojas}]{calo2020ASFEM}
\bibinfo{author}{Calo, V.~M.}, \bibinfo{author}{Ern, A.},
  \bibinfo{author}{Muga, I.}, \& \bibinfo{author}{Rojas, S.}
  (\bibinfo{year}{2020}).
\newblock \bibinfo{title}{An adaptive stabilized conforming finite element
  method via residual minimization on dual discontinuous {G}alerkin norms}.
\newblock {\it \bibinfo{journal}{Comput. Methods Appl. Mech. Engrg.}\/},  {\it
  \bibinfo{volume}{363}\/}, \bibinfo{pages}{112891, 23}.
  \DOIprefix\doi{10.1016/j.cma.2020.112891}.
\bibitem[{Canuto \& Quarteroni(1982)}]{canuto1982approximation}
\bibinfo{author}{Canuto, C.}, \& \bibinfo{author}{Quarteroni, A.}
  (\bibinfo{year}{1982}).
\newblock \bibinfo{title}{Approximation results for orthogonal polynomials in
  {S}obolev spaces}.
\newblock {\it \bibinfo{journal}{Math. Comp.}\/},  {\it
  \bibinfo{volume}{38}\/}, \bibinfo{pages}{67--86}. \URLprefix
  \url{https://doi.org/10.2307/2007465}. \DOIprefix\doi{10.2307/2007465}.
\bibitem[{Chan et~al.(2014)Chan, Heuer, Bui-Thanh \&
  Demkowicz}]{ChaHeuBuiDemCAMWA2014}
\bibinfo{author}{Chan, J.}, \bibinfo{author}{Heuer, N.},
  \bibinfo{author}{Bui-Thanh, T.}, \& \bibinfo{author}{Demkowicz, L.}
  (\bibinfo{year}{2014}).
\newblock \bibinfo{title}{A robust {DPG} method for convection-dominated
  diffusion problems {II}: adjoint boundary conditions and mesh-dependent test
  norms}.
\newblock {\it \bibinfo{journal}{Comput. Math. Appl.}\/},  {\it
  \bibinfo{volume}{67}\/}, \bibinfo{pages}{771--795}.
  \DOIprefix\doi{10.1016/j.camwa.2013.06.010}.
\bibitem[{Cier et~al.(2021)Cier, Poulet, Rojas, Veveakis \&
  Calo}]{cier2020adaptive}
\bibinfo{author}{Cier, R.~J.}, \bibinfo{author}{Poulet, T.},
  \bibinfo{author}{Rojas, S.}, \bibinfo{author}{Veveakis, M.}, \&
  \bibinfo{author}{Calo, V.~M.} (\bibinfo{year}{2021}).
\newblock \bibinfo{title}{Automatically adaptive stabilized finite elements and
  continuation analysis for compaction banding in geomaterials}.
\newblock {\it \bibinfo{journal}{Internat. J. Numer. Methods Engrg.}\/},  {\it
  \bibinfo{volume}{122}\/}, \bibinfo{pages}{6234--6252}. \URLprefix
  \url{https://doi.org/10.1002/nme.6790}. \DOIprefix\doi{10.1002/nme.6790}.
\bibitem[{Cier et~al.(2020{\natexlab{a}})Cier, Rojas \&
  Calo}]{cier2020automatic}
\bibinfo{author}{Cier, R.~J.}, \bibinfo{author}{Rojas, S.}, \&
  \bibinfo{author}{Calo, V.~M.} (\bibinfo{year}{2020}{\natexlab{a}}).
\newblock \bibinfo{title}{An automatic-adaptivity stabilized finite element
  method via residual minimization for heterogeneous, anisotropic
  advection-diffusion-reaction problems}.
\newblock {\it \bibinfo{journal}{arXiv preprint arXiv:2011.11264}\/}, .
\bibitem[{Cier et~al.(2020{\natexlab{b}})Cier, Rojas \&
  Calo}]{Cier2020nonlinear}
\bibinfo{author}{Cier, R.~J.}, \bibinfo{author}{Rojas, S.}, \&
  \bibinfo{author}{Calo, V.~M.} (\bibinfo{year}{2020}{\natexlab{b}}).
\newblock \bibinfo{title}{A nonlinear weak constraint enforcement method for
  advection-dominated diffusion problems}.
\newblock {\it \bibinfo{journal}{Mechanics Research Communications}\/},  (p.
  \bibinfo{pages}{103602}).
\bibitem[{Cohen et~al.(2012)Cohen, Dahmen \& Welper}]{CohDahWelM2AN2012}
\bibinfo{author}{Cohen, A.}, \bibinfo{author}{Dahmen, W.}, \&
  \bibinfo{author}{Welper, G.} (\bibinfo{year}{2012}).
\newblock \bibinfo{title}{Adaptivity and variational stabilization for
  convection-diffusion equations}.
\newblock {\it \bibinfo{journal}{ESAIM Math. Model. Numer. Anal.}\/},  {\it
  \bibinfo{volume}{46}\/}, \bibinfo{pages}{1247--1273}. \URLprefix
  \url{https://doi.org/10.1051/m2an/2012003}.
  \DOIprefix\doi{10.1051/m2an/2012003}.
\bibitem[{Demkowicz \& Gopalakrishnan(2010)}]{Dem2010}
\bibinfo{author}{Demkowicz, L.}, \& \bibinfo{author}{Gopalakrishnan, J.}
  (\bibinfo{year}{2010}).
\newblock \bibinfo{title}{A class of discontinuous {P}etrov-{G}alerkin methods.
  {P}art {I}: the transport equation}.
\newblock {\it \bibinfo{journal}{Comput. Methods Appl. Mech. Engrg.}\/},  {\it
  \bibinfo{volume}{199}\/}, \bibinfo{pages}{1558--1572}.
  \DOIprefix\doi{10.1016/j.cma.2010.01.003}.
\bibitem[{Demkowicz et~al.(2012)Demkowicz, Gopalakrishnan \& Niemi}]{Dem2012}
\bibinfo{author}{Demkowicz, L.}, \bibinfo{author}{Gopalakrishnan, J.}, \&
  \bibinfo{author}{Niemi, A.~H.} (\bibinfo{year}{2012}).
\newblock \bibinfo{title}{A class of discontinuous {P}etrov-{G}alerkin methods.
  {P}art {III}: {A}daptivity}.
\newblock {\it \bibinfo{journal}{Appl. Numer. Math.}\/},  {\it
  \bibinfo{volume}{62}\/}, \bibinfo{pages}{396--427}.
  \DOIprefix\doi{10.1016/j.apnum.2011.09.002}.
\bibitem[{Demkowicz \& Heuer(2013)}]{Dem2013}
\bibinfo{author}{Demkowicz, L.}, \& \bibinfo{author}{Heuer, N.}
  (\bibinfo{year}{2013}).
\newblock \bibinfo{title}{Robust {DPG} method for convection-dominated
  diffusion problems}.
\newblock {\it \bibinfo{journal}{SIAM J. Numer. Anal.}\/},  {\it
  \bibinfo{volume}{51}\/}, \bibinfo{pages}{2514--2537}.
  \DOIprefix\doi{10.1137/120862065}.
\bibitem[{Demkowicz \& Gopalakrishnan(2014)}]{DemGopBOOK-CH2014}
\bibinfo{author}{Demkowicz, L.~F.}, \& \bibinfo{author}{Gopalakrishnan, J.}
  (\bibinfo{year}{2014}).
\newblock \bibinfo{title}{An overview of the discontinuous {P}etrov {G}alerkin
  method}.
\newblock In {\it \bibinfo{booktitle}{Recent developments in discontinuous
  {G}alerkin finite element methods for partial differential equations}\/} (pp.
  \bibinfo{pages}{149--180}).
\newblock \bibinfo{publisher}{Springer, Cham} volume \bibinfo{volume}{157} of
  {\it \bibinfo{series}{IMA Vol. Math. Appl.}\/}.
\newblock \URLprefix \url{https://doi.org/10.1007/978-3-319-01818-8_6}.
  \DOIprefix\doi{10.1007/978-3-319-01818-8\_6}.
\bibitem[{Di~Pietro \& Ern(2012)}]{Ern2012dG}
\bibinfo{author}{Di~Pietro, D.~A.}, \& \bibinfo{author}{Ern, A.}
  (\bibinfo{year}{2012}).
\newblock {\it \bibinfo{title}{Mathematical aspects of discontinuous {G}alerkin
  methods}\/} volume~\bibinfo{volume}{69} of {\it
  \bibinfo{series}{Math\'{e}matiques \& Applications (Berlin) [Mathematics \&
  Applications]}\/}.
\newblock \bibinfo{publisher}{Springer, Heidelberg}.
\newblock \URLprefix \url{https://doi.org/10.1007/978-3-642-22980-0}.
  \DOIprefix\doi{10.1007/978-3-642-22980-0}.
\bibitem[{D\"{o}rfler(1996)}]{dorfler1996}
\bibinfo{author}{D\"{o}rfler, W.} (\bibinfo{year}{1996}).
\newblock \bibinfo{title}{A convergent adaptive algorithm for {P}oisson's
  equation}.
\newblock {\it \bibinfo{journal}{SIAM J. Numer. Anal.}\/},  {\it
  \bibinfo{volume}{33}\/}, \bibinfo{pages}{1106--1124}. \URLprefix
  \url{https://doi.org/10.1137/0733054}. \DOIprefix\doi{10.1137/0733054}.
\bibitem[{Douglas \& Dupont(1976{\natexlab{a}})}]{douglas1976InteriorPenalty}
\bibinfo{author}{Douglas, J., Jr.}, \& \bibinfo{author}{Dupont, T.}
  (\bibinfo{year}{1976}{\natexlab{a}}).
\newblock \bibinfo{title}{Interior penalty procedures for elliptic and
  parabolic {G}alerkin methods}.
\newblock In {\it \bibinfo{booktitle}{Computing methods in applied sciences
  ({S}econd {I}nternat. {S}ympos., {V}ersailles, 1975)}\/} Lecture Notes in
  Phys., Vol. 58 (pp. \bibinfo{pages}{207--216}).
\newblock \bibinfo{publisher}{Springer, Berlin}.
\bibitem[{Douglas \& Dupont(1976{\natexlab{b}})}]{douglas1976intpenalty}
\bibinfo{author}{Douglas, J., Jr.}, \& \bibinfo{author}{Dupont, T.}
  (\bibinfo{year}{1976}{\natexlab{b}}).
\newblock {\it \bibinfo{title}{Interior penalty procedures for elliptic and
  parabolic {G}alerkin methods}\/}.
\newblock Lecture Notes in Phys., Vol. 58.
\newblock \bibinfo{publisher}{Springer, Berlin}.
\bibitem[{Ern \& Guermond(2004)}]{ErnGermond2004}
\bibinfo{author}{Ern, A.}, \& \bibinfo{author}{Guermond, J.-L.}
  (\bibinfo{year}{2004}).
\newblock {\it \bibinfo{title}{Theory and practice of finite elements}\/}
  volume \bibinfo{volume}{159} of {\it \bibinfo{series}{Applied Mathematical
  Sciences}\/}.
\newblock \bibinfo{publisher}{Springer-Verlag, New York}.
\newblock \URLprefix \url{https://doi.org/10.1007/978-1-4757-4355-5}.
  \DOIprefix\doi{10.1007/978-1-4757-4355-5}.
\bibitem[{Feischl et~al.(2016)Feischl, Praetorius \& van~der
  Zee}]{feischl2016abstractAnalGoA}
\bibinfo{author}{Feischl, M.}, \bibinfo{author}{Praetorius, D.}, \&
  \bibinfo{author}{van~der Zee, K.~G.} (\bibinfo{year}{2016}).
\newblock \bibinfo{title}{An abstract analysis of optimal goal-oriented
  adaptivity}.
\newblock {\it \bibinfo{journal}{SIAM J. Numer. Anal.}\/},  {\it
  \bibinfo{volume}{54}\/}, \bibinfo{pages}{1423--1448}. \URLprefix
  \url{https://doi.org/10.1137/15M1021982}. \DOIprefix\doi{10.1137/15M1021982}.
\bibitem[{Giraldo \& Calo(2023)}]{Giraldo2023}
\bibinfo{author}{Giraldo, J.}, \& \bibinfo{author}{Calo, V.}
  (\bibinfo{year}{2023}).
\newblock \bibinfo{title}{An adaptive in space, stabilized finite element
  method via residual minimization for linear \& nonlinear unsteady
  advection-diffusion-reaction equations}.
\newblock {\it \bibinfo{journal}{Mathematical \& Computational Apps}\/},  {\it
  \bibinfo{volume}{28}\/}.
\bibitem[{Gopalakrishnan et~al.(2015)Gopalakrishnan, Monk \&
  Sep\'{u}lveda}]{gopalakrishnan2015}
\bibinfo{author}{Gopalakrishnan, J.}, \bibinfo{author}{Monk, P.}, \&
  \bibinfo{author}{Sep\'{u}lveda, P.} (\bibinfo{year}{2015}).
\newblock \bibinfo{title}{A tent pitching scheme motivated by {F}riedrichs
  theory}.
\newblock {\it \bibinfo{journal}{Comput. Math. Appl.}\/},  {\it
  \bibinfo{volume}{70}\/}, \bibinfo{pages}{1114--1135}. \URLprefix
  \url{https://doi.org/10.1016/j.camwa.2015.07.001}.
  \DOIprefix\doi{10.1016/j.camwa.2015.07.001}.
\bibitem[{Guermond(1999)}]{guermond1999artviz}
\bibinfo{author}{Guermond, J.-L.} (\bibinfo{year}{1999}).
\newblock \bibinfo{title}{Stabilization of {G}alerkin approximations of
  transport equations by subgrid modeling}.
\newblock {\it \bibinfo{journal}{M2AN Math. Model. Numer. Anal.}\/},  {\it
  \bibinfo{volume}{33}\/}, \bibinfo{pages}{1293--1316}.
  \DOIprefix\doi{10.1051/m2an:1999145}.
\bibitem[{Hughes(1995)}]{Hughes1995}
\bibinfo{author}{Hughes, T.~J.} (\bibinfo{year}{1995}).
\newblock \bibinfo{title}{Multiscale phenomena: Green's functions, the
  dirichlet-to-neumann formulation, subgrid scale models, bubbles and the
  origins of stabilized methods}.
\newblock {\it \bibinfo{journal}{Computer Methods in Applied Mechanics and
  Engineering}\/},  {\it \bibinfo{volume}{127}\/}, \bibinfo{pages}{387--401}.
  \DOIprefix\doi{10.1016/0045-7825(95)00844-9}.
\bibitem[{Hughes et~al.(1998)Hughes, Feijóo, Mazzei \& Quincy}]{Hughes1998}
\bibinfo{author}{Hughes, T.~J.}, \bibinfo{author}{Feijóo, G.~R.},
  \bibinfo{author}{Mazzei, L.}, \& \bibinfo{author}{Quincy, J.-B.}
  (\bibinfo{year}{1998}).
\newblock \bibinfo{title}{The variational multiscale method—a paradigm for
  computational mechanics}.
\newblock {\it \bibinfo{journal}{Computer Methods in Applied Mechanics and
  Engineering}\/},  {\it \bibinfo{volume}{166}\/}, \bibinfo{pages}{3--24}.
  \DOIprefix\doi{10.1016/S0045-7825(98)00079-6}.
\newblock \bibinfo{note}{Advances in Stabilized Methods in Computational
  Mechanics}.
\bibitem[{Hughes et~al.(2017)Hughes, Scovazzi \& Franca}]{Hughes:2017}
\bibinfo{author}{Hughes, T. J.~R.}, \bibinfo{author}{Scovazzi, G.}, \&
  \bibinfo{author}{Franca, L.~P.} (\bibinfo{year}{2017}).
\newblock \bibinfo{title}{Multiscale and stabilized methods}.
\newblock In {\it \bibinfo{booktitle}{Encyclopedia of Computational Mechanics
  Second Edition}\/} (pp. \bibinfo{pages}{1--64}).
\newblock \bibinfo{publisher}{John Wiley \& Sons, Ltd}.
\newblock \DOIprefix\doi{10.1002/9781119176817.ecm2051}.
\bibitem[{Johnson et~al.(1984)Johnson, Nävert \&
  Pitkäranta}]{johnson1984SUPG}
\bibinfo{author}{Johnson, C.}, \bibinfo{author}{Nävert, U.}, \&
  \bibinfo{author}{Pitkäranta, J.} (\bibinfo{year}{1984}).
\newblock \bibinfo{title}{Finite element methods for linear hyperbolic
  problems}.
\newblock {\it \bibinfo{journal}{Computer Methods in Applied Mechanics and
  Engineering}\/},  {\it \bibinfo{volume}{45}\/}, \bibinfo{pages}{285--312}.
  \URLprefix
  \url{https://www.sciencedirect.com/science/article/pii/0045782584901580}.
  \DOIprefix\doi{https://doi.org/10.1016/0045-7825(84)90158-0}.
\bibitem[{Keith et~al.(2019)Keith, Vaziri~Astaneh \& Demkowicz}]{keith2019goal}
\bibinfo{author}{Keith, B.}, \bibinfo{author}{Vaziri~Astaneh, A.}, \&
  \bibinfo{author}{Demkowicz, L.~F.} (\bibinfo{year}{2019}).
\newblock \bibinfo{title}{Goal-oriented adaptive mesh refinement for
  discontinuous {P}etrov-{G}alerkin methods}.
\newblock {\it \bibinfo{journal}{SIAM J. Numer. Anal.}\/},  {\it
  \bibinfo{volume}{57}\/}, \bibinfo{pages}{1649--1676}. \URLprefix
  \url{https://doi.org/10.1137/18M1181754}. \DOIprefix\doi{10.1137/18M1181754}.
\bibitem[{Kyburg et~al.(2022)Kyburg, Rojas \& Calo}]{kyburg2020}
\bibinfo{author}{Kyburg, F.~E.}, \bibinfo{author}{Rojas, S.}, \&
  \bibinfo{author}{Calo, V.~M.} (\bibinfo{year}{2022}).
\newblock \bibinfo{title}{Incompressible flow modeling using an adaptive
  stabilized finite element method based on residual minimization}.
\newblock {\it \bibinfo{journal}{Internat. J. Numer. Methods Engrg.}\/},  {\it
  \bibinfo{volume}{123}\/}, \bibinfo{pages}{1717--1735}. \URLprefix
  \url{https://doi.org/10.1002/nme.6912}. \DOIprefix\doi{10.1002/nme.6912}.
\bibitem[{Labanda et~al.(2022)Labanda, Espath \& Calo}]{Labanda2022}
\bibinfo{author}{Labanda, N.~A.}, \bibinfo{author}{Espath, L.}, \&
  \bibinfo{author}{Calo, V.~M.} (\bibinfo{year}{2022}).
\newblock \bibinfo{title}{A spatio-temporal adaptive phase-field fracture
  method}.
\newblock {\it \bibinfo{journal}{Computer Methods in Applied Mechanics and
  Engineering}\/},  {\it \bibinfo{volume}{392}\/}.
  \DOIprefix\doi{10.1016/j.cma.2022.114675}.
\bibitem[{{\L}o{\'s} et~al.(2021){\L}o{\'s}, Rojas, Paszy{\'n}ski, Muga \&
  Calo}]{los2021dgirm}
\bibinfo{author}{{\L}o{\'s}, M.}, \bibinfo{author}{Rojas, S.},
  \bibinfo{author}{Paszy{\'n}ski, M.}, \bibinfo{author}{Muga, I.}, \&
  \bibinfo{author}{Calo, V.~M.} (\bibinfo{year}{2021}).
\newblock \bibinfo{title}{{DGIRM: Discontinuous Galerkin based isogeometric
  residual minimization for the Stokes problem}}.
\newblock {\it \bibinfo{journal}{Journal of Computational Science}\/},  {\it
  \bibinfo{volume}{50}\/}, \bibinfo{pages}{101306}.
\bibitem[{Millar et~al.(2022)Millar, Muga, Rojas \& Van~der Zee}]{millar2021}
\bibinfo{author}{Millar, F.}, \bibinfo{author}{Muga, I.},
  \bibinfo{author}{Rojas, S.}, \& \bibinfo{author}{Van~der Zee, K.~G.}
  (\bibinfo{year}{2022}).
\newblock \bibinfo{title}{Projection in negative norms and the regularization
  of rough linear functionals}.
\newblock {\it \bibinfo{journal}{Numer. Math.}\/},  {\it
  \bibinfo{volume}{150}\/}, \bibinfo{pages}{1087--1121}.
  \DOIprefix\doi{10.1007/s00211-022-01278-z}.
\bibitem[{Niemi et~al.(2011)Niemi, Collier \& Calo}]{Niemi2011}
\bibinfo{author}{Niemi, A.~H.}, \bibinfo{author}{Collier, N.~O.}, \&
  \bibinfo{author}{Calo, V.~M.} (\bibinfo{year}{2011}).
\newblock \bibinfo{title}{Discontinuous petrov-galerkin method based on the
  optimal test space norm for one-dimensional transport problems}.
\newblock {\it \bibinfo{journal}{Procedia Computer Science}\/},  {\it
  \bibinfo{volume}{4}\/}, \bibinfo{pages}{1862--1869}.
  \DOIprefix\doi{10.1016/j.procs.2011.04.202}.
\newblock \bibinfo{note}{Proceedings of the International Conference on
  Computational Science, ICCS 2011}.
\bibitem[{Niemi et~al.(2013{\natexlab{a}})Niemi, Collier \& Calo}]{Niemi2013}
\bibinfo{author}{Niemi, A.~H.}, \bibinfo{author}{Collier, N.~O.}, \&
  \bibinfo{author}{Calo, V.~M.} (\bibinfo{year}{2013}{\natexlab{a}}).
\newblock \bibinfo{title}{Automatically stable discontinuous petrov–galerkin
  methods for stationary transport problems: Quasi-optimal test space norm}.
\newblock {\it \bibinfo{journal}{Computers \& Mathematics with
  Applications}\/},  {\it \bibinfo{volume}{66}\/}, \bibinfo{pages}{2096--2113}.
  \DOIprefix\doi{10.1016/j.camwa.2013.07.016}.
\bibitem[{Niemi et~al.(2013{\natexlab{b}})Niemi, Collier \& Calo}]{Niemi2013b}
\bibinfo{author}{Niemi, A.~H.}, \bibinfo{author}{Collier, N.~O.}, \&
  \bibinfo{author}{Calo, V.~M.} (\bibinfo{year}{2013}{\natexlab{b}}).
\newblock \bibinfo{title}{Discontinuous petrov–galerkin method based on the
  optimal test space norm for steady transport problems in one space
  dimension}.
\newblock {\it \bibinfo{journal}{Journal of Computational Science}\/},  {\it
  \bibinfo{volume}{4}\/}, \bibinfo{pages}{157--163}.
  \DOIprefix\doi{10.1016/j.jocs.2011.07.003}.
\newblock \bibinfo{note}{Agent-Based Simulations, Adaptive Algorithms, ICCS
  2011 Workshop}.
\bibitem[{Poulet et~al.()Poulet, Giraldo, Ramanaidou, Piechocka \&
  Calo}]{Poulet:2023}
\bibinfo{author}{Poulet, T.}, \bibinfo{author}{Giraldo, J.~F.},
  \bibinfo{author}{Ramanaidou, E.}, \bibinfo{author}{Piechocka, A.}, \&
  \bibinfo{author}{Calo, V.~M.} ().
\newblock \bibinfo{title}{Paleo-stratigraphic permeability anisotropy controls
  supergene mimetic martite goethite deposits}.
\newblock {\it \bibinfo{journal}{Basin Research}\/}, .
  \DOIprefix\doi{10.1111/bre.12723}.
\bibitem[{Rojas et~al.(2021)Rojas, Pardo, Behnoudfar \& Calo}]{rojas2021GoA}
\bibinfo{author}{Rojas, S.}, \bibinfo{author}{Pardo, D.},
  \bibinfo{author}{Behnoudfar, P.}, \& \bibinfo{author}{Calo, V.~M.}
  (\bibinfo{year}{2021}).
\newblock \bibinfo{title}{Goal-oriented adaptivity for a conforming residual
  minimization method in a dual discontinuous {G}alerkin norm}.
\newblock {\it \bibinfo{journal}{Comput. Methods Appl. Mech. Engrg.}\/},  {\it
  \bibinfo{volume}{377}\/}, \bibinfo{pages}{Paper No. 113686, 27}.
  \DOIprefix\doi{10.1016/j.cma.2021.113686}.
\bibitem[{Zitelli et~al.(2011)Zitelli, Muga, Demkowicz, Gopalakrishnan, Pardo
  \& Calo}]{zitelli2011}
\bibinfo{author}{Zitelli, J.}, \bibinfo{author}{Muga, I.},
  \bibinfo{author}{Demkowicz, L.}, \bibinfo{author}{Gopalakrishnan, J.},
  \bibinfo{author}{Pardo, D.}, \& \bibinfo{author}{Calo, V.~M.}
  (\bibinfo{year}{2011}).
\newblock \bibinfo{title}{A class of discontinuous {P}etrov-{G}alerkin methods.
  {P}art {IV}: the optimal test norm and time-harmonic wave propagation in
  1{D}}.
\newblock {\it \bibinfo{journal}{J. Comput. Phys.}\/},  {\it
  \bibinfo{volume}{230}\/}, \bibinfo{pages}{2406--2432}. \URLprefix
  \url{https://doi.org/10.1016/j.jcp.2010.12.001}.
  \DOIprefix\doi{10.1016/j.jcp.2010.12.001}.

\end{thebibliography}

\end{document}